\documentclass[12pt,a4paper]{article}

\usepackage{amssymb, amsmath, amsthm}

\usepackage[cm]{fullpage}
\usepackage[english]{babel}
\usepackage[cp1250]{inputenc}
\usepackage[T1]{fontenc}
\usepackage[pdftex]{graphicx}
\usepackage{booktabs}
\usepackage{slashbox}
\usepackage{xcolor}

\usepackage{hyperref}

\newtheorem{thm}{Theorem}
\newtheorem{lem}{Lemma}

\title{Second order scheme for self-similar solutions of a time-fractional porous medium equation on the half-line\footnote{This is an accepted version of the manuscript published in \textit{Applied Mathematics and Computation} 424C (2022), 127033 with DOI: \url{https://doi.org/10.1016/j.amc.2022.127033}}}
\author{Hanna Okrasi{\'n}ska-P{\l}ociniczak\thanks{Department of Mathematics, Wroclaw University of Environmental and Life Sciences, ul. C.K. Norwida 25, 50-275 Wroclaw, Poland}$\;$, \L ukasz P\l ociniczak\thanks{Faculty of Pure and Applied Mathematics, Wroc{\l}aw University of Science and Technology, Wyb. Wyspia{\'n}skiego 27, 50-370 Wroc{\l}aw, Poland}$\;^,$\footnote{Corresponding author. Mail: lukasz.plociniczak@pwr.edu.pl}}
\date{}

\begin{document}
\maketitle
	
\begin{abstract}		
	Many physical, biological, and economical systems exhibit various memory effects due to which their present state depends on the history of the whole evolution. Combined with the nonlinearity of the process these phenomena pose serious difficulties in both analytical and numerical treatment. We investigate a time-fractional porous medium equation that has proved to be important in many applications, notably in hydrology and material sciences. We show that solutions of the free boundary Dirichlet, Neumann, and Robin problems on the half-line satisfy a Volterra integral equation with a non-Lipschitz nonlinearity. Based on this result we prove existence, uniqueness, and construct a family of numerical methods that solve these equations outperforming the usual finite difference approach. Moreover, we prove the convergence of these methods and support the theory with several numerical examples. \\
		
	
\end{abstract}

\section{Introduction}
During the last few decades experimentalists, modellers, theoreticians, and numerical analysts have been investigating many aspects of nonlocality in mathematical objects. These inquiries have certainly been motivated by an increasingly rich experimental evidence for anomalous behaviour, mathematically nontrivial models, and professional curiosity. Apart from many applications, probably the one that is most researched is the so-called anomalous diffusion that arises when the transport is slower or faster than in the classical case \cite{Met00,Kla08}. Probabilistically, this can be characterised with a deviation from the linear mean-square displacement for a randomly walking particle \cite{Kla12}. There is a broad experimental evidence of such dynamics in quantum optics \cite{Sch97}, physics of plasma \cite{Del05}, movement of bacteria and amoeba \cite{Lev97}, G-protein on cell surface \cite{Sun17}, hydrology \cite{De11,Pac03,Plo14, El04, Sun13}, and motion of dislocations in crystal lattice \cite{Hea72}. 

The nonlocality in the anomalous diffusion in usually modelled with a use of the fractional derivative \cite{Met00}. There are many definitions of such operators and it depends on the situation which one is the most suited for a specific model. For example, temporal nonlocality (memory) can be described with Riemann-Liouville or Caputo derivatives while nonlocality in space requires ,,symmetric'' operator such as the fractional Laplacian \cite{Pod98, Caf11, Vaz17}. Physical derivations of the fractional porous media equation have been given for example in \cite{Plo15, Plo19}. We have argued there that the emergence of the time-fractional derivative is a manifestation of the waiting-time phenomenon happening in the medium, i.e. the fluid can be trapped in certain regions for prolonged periods of time which leads to the so-called subdiffusion. Recently, some experimental evidence has been given in \cite{El19} that supports this hypothesis. On the other hand, the spatial nonlocality in the porous medium equation is a consequence of the long jump phenomenon (for a physical derivation see \cite{Plo19}). That is to say, due to medium's complex geometry and heterogeneity fluid parcels can traverse relatively long distances in short amounts of time what in the literature is called superdiffusion. Mathematically this leads to the fractional gradient operator which is closely related to the fractional Laplacian \cite{Caf11, Bil15, Dji19}. A thorough summary of the theory of the fractional porous medium equation can be found \cite{Vaz17}. 

Our previous studies investigated the time-fractional porous medium equation applied as a model for moisture transport in construction materials
\begin{equation}
	\partial^\alpha_t u = \left(u^m u_x\right)_x, \quad 0<\alpha< 1, \quad m>1,
\label{eqn:DiffEqPDE}
\end{equation}
where the fractional derivative is of the Riemann-Liouville type
\begin{equation}
	\partial^\alpha_t u(x,t) = \frac{1}{\Gamma(1-\alpha)}\frac{\partial}{\partial t}\int_0^t (t-s)^{-\alpha} u(x,s) ds.
\end{equation}
with the vanishing initial condition and dry medium at infinity
\begin{equation}
	u(x,0) = 0, \quad u(\infty, t) = 0, \quad x>0, \quad t>0.
\label{eqn:CondPDE}
\end{equation}
The natural function space in which the weak solution of a fractional diffusion equation should be sought is 
\begin{equation}
	C\left([0,1]; H^1(\mathbb{R}_+)\right),
\end{equation}
where $H^1$ is the usual Sobolev space. Below, by the use of integral equation techniques, we show that the self-similar solution of the above problem is indeed continuous and differentiable almost everywhere with respect to space and time. Moreover, in physical derivation of the time-fractional porous medium equation (\ref{eqn:DiffEqPDE}) one arrives at the Caputo derivative on the left-hand side (see for ex. \cite{Plo15}). However, due to the zero initial condition this derivative coincides with the Riemann-Liouville one \cite{Pod98} and hence the present form of (\ref{eqn:DiffEqPDE}). Additionally, motivated by the modelling using continuous time random walk it is interesting to consider some general fractional operators that govern the evolution in time (for some interesting examples and analyses see \cite{Fan21, Fu21}). We leave this topic for future studies.

The parameters $\alpha$ and $m$ are associated with a specific experiment and are subject to fitting with the data. Based on \cite{Sun13} we can think that usually $\alpha\in [0.6,1]$ while $m\approx 7-8$. The parameter $\alpha$ denotes the level of nonlocality of the model with $\alpha = 1$ being the classical case. Since, as will be seen below, our model (\ref{eqn:DiffEqPDE}) describes a slower evolution than in the classical case we speak about subdiffusion.

In previous works we have proved existence and uniqueness of a self-similar solution to the Dirichlet problem \cite{Plo17a}, provided some approximations \cite{Plo14}, solved inverse problem of diffusivity identification \cite{Plo16}, and constructed a numerical scheme \cite{Plo19a}. This numerical method, although only first order accurate, provided a superior performance over the finite difference method applied to the governing PDE. The idea was to use a series of transformations leading to a nonlinear Volterra equation and then to discretize it. In this work we follow a similar route but now we introduce several improvements. First of all, we sketch how to construct an arbitrarily high order scheme and explicitly devise the second order accurate linear method. Then, we prove its convergence under some mild assumptions. The difference from our previous work is that we use the general theory of Volterra equations in order to anticipate the behaviour of its solution. This lets us to conduct yet another transformation which facilitates the crucial step in the convergence proof yielding a stronger result. Moreover, our method is constructed in a way that it gives exact results for power functions which are the leading-order approximations to the solution of the time-fractional porous medium equation. This further improves the accuracy within the second order. For example, as will be shown below, computing the wetting front to several decimal places requires using only moderate step sizes. 

There is a vast literature concerning numerical methods for fractional differential equations. For general surveys one can refer to \cite{Die02, Bal12}. The linear fractional diffusion either with temporal or spatial nonlocality has been solved for example in \cite{Tad06,Li09,Yus05}. The usual approach to constructing a numerical scheme for the time-fractional diffusion is based on discretizing nonlocal derivative by, for example, L1 scheme \cite{Liu07,Lan05}, Gr\"unwald-Letnikov \cite{Yus05}, or Convolution quadrature \cite{Cue06}. This provides the quasi-discrete scheme that can be fully discretized by introducing a finite difference \cite{Lin07}, finite element \cite{Zhu16} or other spatial discretization \cite{Li09}. 

On the other hand, the case of nonlinear anomalous diffusion has been rigorously treated only in relatively few papers (especially when the diffusivity is a function of the solution). Some recent examples include time-fractional reaction-diffusion equations \cite{Li19c}. One of the very useful tools in proving convergence for a nonlinear problem is the discrete fractional Gr\"onwall inequality that was developed and applied to nonlinear time-fractional diffusion problems in \cite{Li18,Lia18} (note that probably the first account of fractional Gr\"onwall inequality has been given in \cite{McK82} but see also \cite{Web19} and references therein). Numerical methods for equations with non-smooth data have been also analysed in \cite{Zak20}. On the other hand, the space-fractional porous medium has been discretized in \cite{Del14} with the use of finite difference method applied to the extension problem \cite{Caf07}. Furthermore, in \cite{Del18} some very robust numerical methods have been devised, thoroughly analysed, and applied to a very broad classes of nonlocal equations. Moreover, authors of \cite{Gas17} used a multi-grid waveform method for the time-fractional case. The nonlinear source terms have been considered in \cite{Bhr16}. The main difficulty in the analysis of the porous medium equation (even in the classical case) is the degeneracy of diffusivity, i.e. the coefficient of $u_x$ can vanish. This fact makes both theoretical and numerical studies interesting and non-trivial. For several years many approaches of numerical treatment of the classical porous medium equation have been proposed. To paint the background we mention only a few papers. For example, in \cite{Dib84} authors used a finite difference method with an interface tracking algorithm. This feature is crucial since due to degeneracy in many circumstances the solution is characterized with a finite speed of propagation what translates into a free-boundary problem \cite{Sta14}. Also, a finite element method for the classical porous medium equation has been analysed in \cite{Arb96,Ebm98}. This topic is still vigorously investigated and mathematicians constantly develop some new approaches (for ex. \cite{Eti12,Pop02,Eis18}).

The root of our method lies in appropriate discretization of a nonlinear Volterra equation (for a thorough treatment see \cite{Bru17}). A survey of numerical methods for such problems can be found in \cite{Bak00,Lin85}. Notice that the majority of the theory concerns Lipschitz nonlinearity while our model fails to satisfy this condition. There are only a few papers that developed rigorous methods for nonlipschitzian case \cite{Fri97}. Notably, in \cite{Buc97} an iterative scheme has been proposed. We also mention our previous work \cite{Plo17} where an idea for a proof similar to the one used in the present paper originated. Finally, in \cite{Cor00,Ari19} a highly readable theoretical summaries about non-Lipschitz problems have been published. 

As we mentioned before, this paper concerns a high order numerical method for finding self-similar solutions of (\ref{eqn:DiffEqPDE}). In Section \ref{sec:SelfSimilar} we perform a series of transformations that reduce the main nonlocal nonlinear PDE into an ordinary Volterra equation with non-Lipschitz nonlinearity. As a side result we prove a existence and uniqueness result about self-similar solutions on the half-line with Dirichlet, Neumann, and Robin conditions. Next, in Section \ref{sec:NumericalMethod} we present a general scheme of numerical method construction along with explicit formulas for a second order linear interpolation scheme (trapezoidal). We also prove that it is convergent provided that the parameter $m$ is sufficiently large. In Section \ref{sec:NumericalExamples} we illustrate our theory with numerical examples concerning some exemplary equations and the time-fractional porous medium equation itself. We close the paper with some concluding remarks.

\section{Self-similar solutions and Volterra equations}\label{sec:SelfSimilar}
In applications one is frequently interested in finding self-similar solutions of (\ref{eqn:DiffEqPDE}) in the form
\begin{equation}
	u(x,t) = t^{a} U(\eta), \quad \eta = x t^{-b},
\label{eqn:SelfSimialarVars}
\end{equation}
where $a,b \geq 0$ are constants to be found and the similarity variable is denoted by $\eta$. Our main motivation for looking for such solutions lies in the fact that in the real-world experiments of moisture percolation in porous media one observes the self-similar profile (see \cite{El04, El19, Pac03, Sun13}). On the other hand, in a large generality Lie group methods can be used to look for many families of self-similar solutions in a systematic way \cite{Gaz09, Gaz13, Gor00}. When we plug the above into the original PDE (\ref{eqn:DiffEqPDE}) we transform the spatial
\begin{equation}
	\left( u^m(x,t) u_x(x,t) \right)_x = t^{a(m+1)-2b} \frac{d}{d\eta}\left(U^m(\eta) \frac{d}{d\eta}U(\eta)\right).
\end{equation}
and the temporal derivatives
\begin{equation}
\begin{split}
	\partial^\alpha_t u(x,t)&=\frac{\partial}{\partial t} \left( I^{1-\alpha}_t \left(t^a U(x t^{-b})\right) \right)=\frac{1}{\Gamma(1-\alpha)}\frac{\partial}{\partial t} \int_0^t (t-z)^{-\alpha} z^a U(x z^{-b})dz\\
	&=\frac{1}{\Gamma(1-\alpha)}\frac{\partial}{\partial t} \left(t^{a-\alpha+1} \int_0^1 (1-s)^{-\alpha} s^a U(\eta s^{-b})ds\right),
\end{split}
\end{equation}
If we further introduce the Erd\'{e}lyi-Kober fractional integral operator (see \cite{Kir93})
\begin{equation}
	I^{\beta,\gamma}_\delta U(\eta) = \frac{1}{\Gamma(\gamma)}\int_0^1 (1-z)^{\gamma-1} z^\beta U(\eta z^\frac{1}{\delta}) dz,
\label{eqn:EK}
\end{equation}
we can write (\ref{eqn:DiffEqPDE}) as 
\begin{equation}
	t^{\alpha+am - 2b}\frac{d}{d\eta}\left(U^m \frac{dU}{d\eta}\right) = \left[(1-\alpha+a)-b\eta \frac{d}{d\eta}\right]I^{a, 1-\alpha}_{-\frac{1}{b}} U, \quad 0<\alpha< 1,
\end{equation}
Now, the self-similar solution is possible only if
\begin{equation}
	2b - m a = \alpha,
\label{eqn:Cond1}
\end{equation}
which is the consistency condition. Since there are two unknowns: $a=a(\alpha,m)$ and $b=b(\alpha,m)$ we need yet another equation. This comes from the initial-boundary conditions and we assume that initially the medium is dry (\ref{eqn:CondPDE}) which expressed in the self-similar variables (\ref{eqn:SelfSimialarVars}) yields
\begin{equation}
	U(\infty) = 0.
\label{eqn:InitialCondSelfSimilar}
\end{equation}
Furthermore, for the boundary conditions we have several self-similar physically relevant choices. For example, we can consider the Dirichlet
\begin{equation}
	u(0,t) = 1, \quad t>0,
\label{eqn:Dirichlet}
\end{equation}
which describes constant concentration, Neumann 
\begin{equation}
	u^m(0,t) u_x(0,t) = -1, \quad t>0,
\label{eqn:Neumann}
\end{equation}
describing constant flux, or Robin problem
\begin{equation}
	u^m(0,t)u_x(0,t) = -u(0,t), \quad t>0,
\label{eqn:Robin}
\end{equation}
which translates into requirement that the flux is proportional to the concentration at the interface $x = 0$. For a physical account of these models see \cite{Plo15}. Boundary conditions provide us with the second equation for $a$ and $b$ and eventually we are able to write (\ref{eqn:DiffEqPDE}) as
\begin{equation}
	\frac{d}{d\eta}\left(U^m \frac{dU}{d\eta}\right) = A I^{a, 1-\alpha}_{-\frac{1}{b}} U - B \eta \frac{d}{d\eta} I^{a, 1-\alpha}_{-\frac{1}{b}} U, \quad 0<\alpha< 1,
\label{eqn:SelfSimilarEq}
\end{equation}
for some $A=A(\alpha,m)$ and $B=B(\alpha,m)$. The exact values of these constants for different boundary conditions are summarized in Tab. \ref{tab:AB}. Note that always $A,B\geq 0$. 

\begin{table}
	\centering
	\begin{tabular}{rcccc}
		\toprule
		Condition & $a$ & $A$ & $b=B$ & $\eta^*$  \\
		\midrule
		Dirichlet (\ref{eqn:Dirichlet}) & 0 & $1-\alpha$ & $\dfrac{\alpha}{2}$ & $\dfrac{1}{\sqrt{y(1)^m}}$  \vspace{5pt}\\
		Neumann (\ref{eqn:Neumann}) & $\dfrac{\alpha}{m+2}$ & $1-\dfrac{m+1}{m+2}$ & $\dfrac{m+1}{m+2}\alpha$ & $\left(\dfrac{m+1}{(y^{m+1})'(1)}\right)^{\frac{m}{m+2}}$ \vspace{5pt} \\
		Robin (\ref{eqn:Robin}) & $\dfrac{\alpha}{m}$ & $1-\dfrac{m-1}{m}\alpha$ & $\alpha$ & $\dfrac{m}{(y^m)'(1)}$ \\
		\bottomrule
	\end{tabular}
	\caption{Values of $a$, $b$ in (\ref{eqn:SelfSimialarVars}), $A=A(\alpha,m)$, $B=B(\alpha,m)$ as in (\ref{eqn:SelfSimilarEq}), and the wetting front $\eta^*$ defined in (\ref{eqn:CompactSupport}) for different types of boundary conditions. }
	\label{tab:AB}
\end{table}

To motivate our further reasoning let us focus on the Robin problem (\ref{eqn:Robin}) for which
\begin{equation}
	a = \frac{\alpha}{m}, \quad b = \alpha.
\label{eqn:ab}
\end{equation} 
Other conditions have been considered in \cite{Plo15,Plo17a} and lead to a similar final integral equation. In any of the considered boundary conditions the support of the solution will be compact (for a proof see \cite{Dji18a} and for the classical case \cite{Atk71}). Physically, this fact is equivalent to a finite speed of wetting front propagation and arises in the degeneracy of the equation (\ref{eqn:DiffEqPDE}). Having that in mind, let $\eta^*>0$ be such that
\begin{equation}
	U(\eta) = 0 \quad \text{for} \quad \eta>\eta^*.
\label{eqn:CompactSupport}
\end{equation}
Moreover, by integrating (\ref{eqn:SelfSimilarEq}) one can show that there is no flux through the support's boundary (see \cite{Plo17a})
\begin{equation}
	U(\eta)^m \frac{d U}{d\eta}(\eta) = 0 \quad \text{for} \quad \eta>\eta^*.
\label{eqn:CompactSupportFlux}
\end{equation}
Note that $\eta^*$ is a-priori unknown and has to be determined as a part of the solution since we are dealing with a free-boundary problem. A standard way of proceeding is to introduce yet another transformation
\begin{equation}
	U(\eta) = C y(z), \quad z = 1-\frac{\eta}{\eta^*}, \quad z\geq 0,
\label{eqn:TransformationY}
\end{equation}
where $C$ has to be determined. The conditions (\ref{eqn:CompactSupport})-(\ref{eqn:CompactSupportFlux}) now become
\begin{equation}
	y(0) = 0, \quad y^m(0) y'(0) =0,
\label{eqn:InitialConditionY}
\end{equation}
where prime denotes the differentiation with respect to $z$. The Robin condition (\ref{eqn:Robin}) now yields
\begin{equation}
	U^m(0) \frac{dU}{d\eta} (0) = - U(0) \rightarrow (y^m)'(1) = \frac{m\eta^*}{C^m},
\label{eqn:RobinY}
\end{equation}
which gives us one equation to determine $\eta^*$ and $C$. 

When we plug (\ref{eqn:TransformationY}) into (\ref{eqn:SelfSimilarEq}) we obtain
\begin{equation}
	\frac{C^m}{(\eta^*)^2}  \left(y^m y'\right)' = A G_\alpha y + B(1-z) (G_\alpha y)',
\label{eqn:EquationY}
\end{equation}
where for $a$ and $b$ as in (\ref{eqn:ab}) we have defined
\begin{equation}
\begin{split}
	G_\alpha y(z) &= \frac{1}{\Gamma(1-\alpha)} \int_{(1-z)^{1/b}}^1 (1-s)^{-\alpha} s^a y(1-s^{-b}(1-z))ds \\
	&= \frac{1}{b} \frac{1}{\Gamma(1-\alpha)} \left(1-z\right)^{\frac{a+1}{b}} \int_0^{z} \left(1-\left(\frac{1-z}{1-w}\right)^\frac{1}{b}\right)^{-\alpha} \frac{y(w)}{(1-w)^{1+\frac{a+1}{b}}}dw.
\end{split}
\label{eqn:G}
\end{equation}
The second equality above follows from the change of variables $w = 1-s^{-a}(1-z)$. We can thus see that 
\begin{equation}
	C = \left(\eta^*\right)^\frac{2}{m},
\label{eqn:CEqn}
\end{equation}
which together with (\ref{eqn:RobinY}) closes the system and gives us an equation for $\eta^*$ provided we know the solution $y$, that is $\eta^*$ can be found from
\begin{equation}
	\eta^* = \frac{m}{(y^m)'(1)}.
\end{equation}

Now, (\ref{eqn:EquationY}) can be transformed into a nonlinear integral Volterra equation. First, with an integration and using the condition (\ref{eqn:InitialConditionY}) we obtain
\begin{equation}
\begin{split}
	y(z)^m y'(z) 
	&= A \int_0^z G_\alpha y(s)ds + B \int_0^z \left (1-z)(G_\alpha y(s)\right)' ds \\
	&= B (1-z)G_\alpha y(z) + (A+B)\int_0^z G_\alpha y(s)ds,
\end{split}
\label{eqn:yy'}
\end{equation}
where we have integrated by parts. A second integration finally yields
\begin{equation}
\begin{split}
	y(z)^{m+1} &= (m+1) \int_0^z \left(B(1-s)+(A+B)\left(z-s\right)\right) G_\alpha y(s) ds \\
	&=: \int_0^z F(z,s) G_\alpha y(s) ds.
\end{split}
\label{eqn:VolterraEquationF}
\end{equation}
Next, we use the definition of $G_\alpha$ operator (\ref{eqn:G}) and change the order of integration (Fubini's theorem)
\begin{equation}
	y(z)^{m+1} = \frac{1}{b}\frac{1}{\Gamma(1-\alpha)}\int_0^{z} \left(\int_{u}^z F(z,s) (1-s)^{\frac{a+1}{b}} \left(1-\left(\frac{1-s}{1-u}\right)^\frac{1}{b}\right)^{-\alpha}ds\right) \frac{y(u)du}{(1-u)^{1+\frac{a+1}{b}}}.
\end{equation}
We can now substitute back $v = ((1-s)/(1-u))^{1/b}$ inside the inner integral to obtain
\begin{equation}
\begin{split}
	y(z)^{m+1} 
	&= \frac{1}{\Gamma(1-\alpha)} \int_0^{z} \left[ \int_{\left(\frac{1-z}{1-u}\right)^{1/b}}^{1} F(z, 1-s^b(1-u)) (1-s)^{-\alpha} s^{a+b} ds\right] y(u) du \\
	&=: \int_{0}^z K(z,u) y(u) du,
\end{split}
\label{eqn:VolterraRobin}
\end{equation}
where we have defined the kernel $K$. Observe that it can be written as a combination of incomplete beta functions (and we use this fact later). Moreover, the kernel is a continuous function since $(1-s)^{-\alpha}$ is integrable near $s=1$ and $a>0$. We can also obtain some simple bounds on $K$ due to the fact that $F(z,s)$ is a linear function in its both variables. Assume that $0\leq z\leq X$ for some $0<X<1$. Then,
\begin{equation}
	B(1-X) \leq F(z,s) \leq A+2B,
\end{equation}
and hence
\begin{equation}
	K(z,u) \leq \frac{m+1}{\Gamma(1-\alpha)} (A+2B) \int_{\left(\frac{1-z}{1-u}\right)^{1/b}}^1 (1-s)^{-\alpha} ds = \frac{m+1}{\Gamma(2-\alpha)} (A+2B) \left(1-\left(\frac{1-z}{1-u}\right)^{1/b}\right)^{1-\alpha}.
\end{equation}
Now, we have
\begin{equation}
	\frac{1-z}{1-s} = 1 - \frac{z-s}{1-s},
\end{equation}
what together with elementary estimate $(1-x)^\gamma \geq 1-\gamma x$ for $\gamma\geq 1$ leads to
\begin{equation}
	K(z,u) \leq \frac{m+1}{\Gamma(2-\alpha)} (A+2B) \left(\frac{1}{b}\frac{z-u}{1-X}\right)^{1-\alpha} \leq K_+ (z-u)^{1-\alpha},
\label{eqn:K+}
\end{equation}
because $0\leq u\leq z\leq X$. Similarly, we can obtain bounds from below. To this end, notice that 
\begin{equation}
\begin{split}
	K(z,u) 
	&\geq \frac{m+1}{\Gamma(1-\alpha)} B(1-X) \left(\frac{1-z}{1-u}\right)^{a/b+1} \int_{\left(\frac{1-z}{1-u}\right)^{1/b}}^1 (1-s)^{-\alpha} ds \\
	&= \frac{m+1}{\Gamma(2-\alpha)} B(1-X) \left(\frac{1-z}{1-u}\right)^{a/b+1} \left(1-\left(\frac{1-z}{1-u}\right)^{1/b}\right)^{1-\alpha}.
\end{split}
\end{equation}
Using the convexity bound $(1-x)^\gamma \leq 1-x$ for $\gamma\geq 1$ we now have
\begin{equation}
	K(z,u) \geq \frac{m+1}{\Gamma(2-\alpha)} B(1-X) \left(\frac{1-z}{1-u}\right)^{a/b+1} \left(\frac{z-u}{1-u}\right)^{1-\alpha}\geq K_- (z-u)^{1-\alpha},
\end{equation} 
since $0\leq u\leq z\leq X$. Whence, we have found the behaviour of the kernel yielding
\begin{equation}
	K_- (z-u)^{1-\alpha} \leq K(z,u) \leq K_+ (z-u)^{1-\alpha},
\label{eqn:KernelBounds}
\end{equation}
where $K_\pm$ are known constants. These kernel bounds have important consequence for the form of the solution and are crucial for our subsequent construction of the numerical method. For example, as was shown in \cite{Buc05} (for a further account see also \cite{Gri81,Bus90,Okr95}) the equation (\ref{eqn:VolterraRobin}) has only one non-trivial solution which satisfies
\begin{equation}
	C_-^\frac{1}{m} z^\frac{2-\alpha}{m} \leq y(z) \leq C_+^\frac{1}{m} z^\frac{2-\alpha}{m},
\label{eqn:YBounds}
\end{equation}
where $C_\pm>0$ can be found explicitly, however, their exact values are not relevant for our reasoning. The calculations for other boundary conditions presented in Tab. \ref{tab:AB} are essentially identical. Therefore, we have shown the following result.
\begin{thm}
	Let
	\begin{equation}
		u(x,t) = (\eta^*)^\frac{2}{m} t^a y \left(1-\frac{1}{\eta^*}\frac{x}{t^b}\right),
	\end{equation}
	where $y=y(z)$ is a unique nontrivial solution of the Volterra equation (\ref{eqn:VolterraRobin}) while $a$, $b$, and $\eta^*$ are given in Tab. \ref{tab:AB}. Then, $u$ is a unique self-similar solution of (\ref{eqn:DiffEqPDE}) with vanishing initial condition along with one of the self-similar boundary conditions (\ref{eqn:Dirichlet})-(\ref{eqn:Robin}). Moreover, we have the following estimates
	\begin{equation}
		U_- \left(1-\frac{1}{\eta^*}\frac{x}{t^b}\right)^\frac{2-\alpha}{m} \leq u(x,t) \leq U_+ \left(1-\frac{1}{\eta^*}\frac{x}{t^b}\right)^\frac{2-\alpha}{m},
	\end{equation}
	with suitable constants $U_\pm>0$ that can be found explicitly. 
\end{thm}

For example, when we consider the Dirichlet boundary condition and consult Tab. \ref{tab:AB} we obtain
\begin{equation}
	K_- = \frac{m+1}{\Gamma(2-\alpha)} \frac{\alpha}{2}(1-X)^{\alpha-1}, \quad K_+ = \frac{m+1}{\Gamma(2-\alpha)} \left(\frac{\alpha}{2}(1-X)\right)^{\alpha-1},
\end{equation}
which by results from \cite{Buc05} yield
\begin{equation}
	C_\pm = K_\pm \beta\left(2-\alpha, 1+\frac{2-\alpha}{m}\right),
\end{equation}
where $\beta$ is the Euler beta function. Eventually, the solution of (\ref{eqn:DiffEqPDE}) can be bounded by
\begin{equation}
	U_\pm = \frac{1}{y(1)} C_\pm^\frac{1}{m},
\end{equation}
where $y(1)$ can be calculated from the solution of (\ref{eqn:VolterraRobin}). 

\section{Numerical method}\label{sec:NumericalMethod}
In the previous section we have shown that looking for solutions of (\ref{eqn:DiffEqPDE}) with self-similar boundary conditions (\ref{eqn:Dirichlet})-(\ref{eqn:Robin}) is equivalent to solving (\ref{eqn:VolterraRobin}). Therefore, in what follows we will devise a numerical method for solving a general class of nonlinear Volterra equations of the form 
\begin{equation}
	y(z)^{m+1} = \int_0^z K(z,s) y(s) ds, \quad 0\leq z\leq 1,
\label{eqn:VolterraY}
\end{equation}
where we assume that the kernel is continuous and there exist constants $K_\pm$ such that
\begin{equation}
	 K_- (z-s)^\gamma \leq K(z,s) \leq K_+ (z-s)^\gamma, \quad \gamma\geq 0.
\label{eqn:BoundY}
\end{equation}
Note that in this section we use the same letters for $y$ and $K$ as before, however now we are considering a \emph{general} Volterra equation which may not have anything in common with anomalous diffusion. By results from \cite{Buc05} the above equation has a unique non-trivial solution satisfying
\begin{equation}
	C_-^\frac{1}{m} z^\frac{\gamma+1}{m} \leq y(z) \leq C_+^\frac{1}{m} z^\frac{\gamma+1}{m},
\end{equation}
for suitable constants $C_\pm>0$. Therefore, the behaviour of the solution is of power type and for a numerical treatment it is reasonable to peel it off from the actual form of $y$. That is to say, we substitute
\begin{equation}
	y(z) = z^\frac{\gamma+1}{m} v(z),
\end{equation}
which leads to an equivalent integral equation
\begin{equation}
	v(z) = z^{-\frac{(m+1)(\gamma+1)}{m}} \int_0^z K(z,s) s^\frac{\gamma+1}{m} v(s) ds.
\label{eqn:VolterraEqV}
\end{equation}
It is crucial to note that $v$ is now bounded away from zero, that is
\begin{equation}
	0 < C_-^\frac{1}{m} \leq v(z) \leq C_+^\frac{1}{m},
\label{eqn:BoundsV}
\end{equation}
what facilitates both the analysis and numerical computations. A particular choice of the kernel, i.e. when $K(z,s) = K_+ (z-s)^\gamma$ clearly leads to a constant solution
\begin{equation}
	v(z) = \left(K_+ \beta\left(\gamma+1,\frac{\gamma+1}{m}+1\right)\right)^\frac{1}{m},
\label{eqn:ConstantSolution}
\end{equation}
where $\beta$ is Euler beta function. This simple solution can serve as a benchmark of various numerical methods. 

\subsection{Construction}
Now we can proceed to the discretization. Introduce the grid
\begin{equation}
	z_n := \frac{n}{N}, \quad h:=\frac{1}{N}, \quad n=0,1,...,N,
\end{equation}
where the total number of grid points $N$ is fixed. Further, we can discretize the integral in (\ref{eqn:VolterraEqV})
\begin{equation}
	\int_0^{z_n} K(z_n,s) s^\frac{\gamma+1}{m} v(s) ds = \sum_{i=0}^{n-1} \left(w_{n,i}(h) v(z_i) + \delta_i(h)\right),
\label{eqn:Quadrature}
\end{equation}
where $\delta_i(h)$ is the local consistency error, and $w_{n,i}(h)$ are weights that depend on the form of the kernel. When we use (\ref{eqn:Quadrature}) in (\ref{eqn:VolterraEqV}) and truncate the remainder we obtain the following numerical scheme
\begin{equation}
	v_n^{m+1} = z_n^{-\frac{(m+1)(\gamma+1)}{m}} \sum_{i=0}^{n-1} w_{n,i}(h) v_i,
\label{eqn:NumMet}
\end{equation}
where by $v_i$ we have denoted the numerical approximation to $v(z_i)$ while $K_{n,i} = K(z_n,z_i)$. Note also that in the quadrature (\ref{eqn:Quadrature}) we have not included the last point of the interval $[0,z_n]$. This leads to a (half-)open quadrature and has been chosen in order to reduce the computational cost of the method. For if the sum on the right-hand side of the above included $v_n$ term we would obtain an implicit method that would require solving a nonlinear equation of the form $x^{m+1}-a_1 x + a_0 = 0$ in each step. In order to keep the cost as low as possible we consider only explicit methods. 

Different quadratures would yield different values of the weights $w_{n,i}(h)$. For example, the simplest rectangular rule in which the whole integrand is approximated by its value at the lower terminal would yield
\begin{equation}
	w_{n,i}(h) = h K(z_n,z_i) z_i^{\frac{\gamma+1}{m}}.	
\label{eqn:RectangleSimplest}
\end{equation}
This method could be proved to be convergent (see \cite{Plo19a}), however, it does not solve the constant case exactly, that is to say when $K(z,s) = K_+ (z-s)^\gamma$ the numerical solution $v_n$ is not equal to (\ref{eqn:ConstantSolution}). To see this suppose that $v_0 = v_1 = v_2 = C$, then from (\ref{eqn:NumMet}) at $n=2$ and (\ref{eqn:RectangleSimplest})
\begin{equation}
	C^{m+1} = C K_+ h^{-\frac{(m+1)(\gamma+1)}{m}+1} (2h - h)^\gamma h^{\frac{\gamma+1}{m}} = C K_+,
\end{equation}
hence $C = (K_+)^{1/m}$. Then, the third step yields
\begin{equation}
	v_3^{m+1} = K_+ h^{-\frac{(m+1)(\gamma+1)}{m}+1} \left((3h-h)^\gamma h^{\frac{\gamma+1}{m}} + (3h-2h)^\gamma (2h)^{\frac{\gamma+1}{m}}\right)= K_+ \left(2^\gamma + 1\right) \neq C^{m+1},
\end{equation}
for $\gamma > 0$. Therefore, $v_n$ cannot be constant for all $n \geq 3$. An ability of a numerical scheme to be able to solve for a constant solution can be thought as a necessary requirement that we have to make since it would accurately resolve the zero order Taylor series term. A family of such methods can be devised by using an interpolating polynomial for the unknown $v_n$ in the integral (\ref{eqn:Quadrature}). The kernel, since it is known, is not approximated. Although in this work we will focus only on the first degree interpolation, i.e. a linear reconstruction, it is instructive to see how does the zeroth order approximation look like. To this end we write
\begin{equation}
	\int_0^{z_n} K(z_n,s) s^\frac{\gamma+1}{m} v(s) ds = \sum_{i=0}^{n-1} \int_{z_i}^{z_{i+1}} K(z_n,s) s^{\frac{\gamma+1}{m}} v(s) ds,
\end{equation}
in which we expand $v$ in its Taylor series $v(s) = v(z_i) + v'(\widehat{\xi}_i) (s-z_i)$, which leads to
\begin{equation}
	\int_0^{z_n} K(z_n,s) s^\frac{\gamma+1}{m} v(s) ds = \sum_{i=0}^{n-1} \left[\left(\int_{z_i}^{z_{i+1}} K(z_n,s) s^{\frac{\gamma+1}{m}}ds\right) v_i + \delta_i(h)\right],
\end{equation}
where by the mean value theorem the remainder is
\begin{equation}
	\delta_i(h) = v'(\xi_i) \int_{z_i}^{z_{i+1}} K(z_n,s) s^{\frac{\gamma+1}{m}} (s-z_i)ds.
\end{equation}
Therefore, going back to (\ref{eqn:VolterraEqV}) gives
\begin{equation}
	v(z_n)^{m+1} = z_n^{-\frac{(m+1)(\gamma+1)}{m}} \sum_{i=0}^{n-1} \left(\int_{z_i}^{z_{i+1}} K(z_n,s) s^{\frac{\gamma+1}{m}}ds\right) v_i + \delta_n(h),
\end{equation}
where the remainder satisfies
\begin{equation}
\begin{split}
	|\delta_n(h)| \leq z^{-\frac{(m+1)(\gamma+1)}{m}} \sum_{i=0}^{n-1} |\delta_i(h)| &=
	z_n^{-\frac{(m+1)(\gamma+1)}{m}} |v'(\xi)| \sum_{i=0}^{n-1} \int_{z_i}^{z_{i+1}} K(z_n,s) s^{\frac{\gamma+1}{m}} (s-z_i)ds \\ &\leq 
	z_n^{-\frac{(m+1)(\gamma+1)}{m}} K_+ |v'(\xi)| h \sum_{i=0}^{n-1} \int_{z_i}^{z_{i+1}} (z_n-s)^\gamma s^{\frac{\gamma+1}{m}} ds.
\end{split}
\end{equation}
where, once again, we have used the mean value theorem and used (\ref{eqn:KernelBounds}). Further, with a substitution $s = nh \sigma$ we can evaluate the integral
\begin{equation}
	|\delta_n(h)| \leq K_+ \|v'(\xi)\| \beta\left(\gamma+1, \frac{\gamma+1}{m}\right)h = O(h), \quad h\rightarrow 0^+,
\end{equation}
uniformly for $n\in \mathbb{N}$. Whence, the truncation error of the method is of the first order. Here, $\beta(\cdot,\cdot)$ is the Euler beta function. Neglecting the remainder we obtain a rectangle method for solving (\ref{eqn:VolterraEqV})
\begin{equation}
	v_n^{m+1} = z_n^{-\frac{(m+1)(\gamma+1)}{m}} \sum_{i=0}^{n-1} \left(\int_{z_i}^{z_{i+1}} K(z_n,s) s^{\frac{\gamma+1}{m}}ds\right) v_i.
\label{eqn:Rectangle}
\end{equation}
In contrast with (\ref{eqn:RectangleSimplest}) the above method solves for the constant solution exactly by the very construction (in that case $v'\equiv 0$ and the remainder vanishes). 

It is now a straightforward exercise to develop higher order methods with the use of the higher degree of Lagrange's interpolating polynomial. We will analyse only the second order method since we are able to equip it with a suitable choice of initial conditions. As we shall see below, for higher order methods a prescription of starting values is not a straightforward task. This situation is similar to multistep methods for ODEs. Having that in mind we can use a linear approximation to $v$ in the integral in (\ref{eqn:Quadrature}). The important point is to construct an explicit method by not including the terminal point of the interval into the interpolation. That is to say, we have to partition the interval $[0,z_n]$ into subintervals which contain exactly three nodes. Then, we construct a linear interpolation based on first two nodes. Due to this construction we have to consider separately cases when $n$ is even or odd.

First, suppose that $n$ is even. Then, we build linear approximations based on the first and second points of a three node interval starting with $z = 0$. That is to say, we have the partition
\begin{equation}
	\int_0^{z_n} K(z_n,s) s^\frac{\gamma+1}{m} v(s) ds = \sum_{i=0}^{\frac{n}{2}-1} \int_{z_{2i}}^{z_{2i+2}} K(z_n,s) s^{\frac{\gamma+1}{m}} v(s) ds,
\end{equation}
in which we use Lagrange's interpolation 
\begin{equation}
	v(s) = \left(1-\frac{s-z_{2i}}{h}\right)v(z_{2i}) + \frac{s-z_{2i}}{h} v(z_{2i+1}) + \frac{1}{2}v''(\widehat{\xi}_i) (s-z_{2i})(s-z_{2i+1}).
\end{equation}
Here, $\widehat{\xi}_i$ is the intermediate point needed in the remainder. If we use that in the integral we obtain
\begin{equation}
\begin{split}
	\int_0^{z_n} K(z_n,s) s^\frac{\gamma+1}{m} v(s) ds 
	&= \sum_{i=0}^{\frac{n}{2}-1} w_{n,2i}(h) v(z_{2i}) + w_{n,2i+1}(h) v(z_{2i}) \\
	&+\frac{1}{2}v''(\xi) \sum_{i=0}^{\frac{n}{2}-1} \int_{z_{2i}}^{z_{2i+2}} K(z_n,s) s^{\frac{\gamma+1}{m}}(s-z_{2i})(s-z_{2i+1}) ds,
\end{split}
\label{eqn:SecondOrderEven}
\end{equation}
where we have utilized the mean value theorems for integrals and sums. The weights are given by
\begin{equation}
	w_{n,i}(h) = \int_{z_i}^{z_{i+2}} K(z_n,s) s^{\frac{\gamma+1}{m}}
	\left( 
	\begin{cases}
		1-\frac{s-z_{i}}{h}, & i \text{ even}\\
		\frac{s-z_{i}}{h}, & i \text{ odd}
	\end{cases}
	\right) ds
	\quad \text{for } n \text{ even and } 0\leq i\leq n-2.
\label{eqn:WeightsEven}
\end{equation}
A similar calculation can be conducted for odd values of $n$. In that case, however, we first linearly interpolate between $z_0$ and $z_1$ and then partition the rest of the interval $[z_1,z_n]$ into three node subintervals. We then have
\begin{equation}
	\int_0^{z_n} K(z_n,s) s^\frac{\gamma+1}{m} v(s) ds = \int_0^{z_1} K(z_n,s) s^\frac{\gamma+1}{m} v(s) ds + \sum_{i=1}^{\frac{n-1}{2}}\int_{z_{2i-1}}^{z_{2i+1}} K(z_n,s) s^\frac{\gamma+1}{m} v(s) ds.
\end{equation}
Further, conducting analogous interpolation as in the even case leads to
\begin{equation}
\begin{split}
	\int_0^{z_n} & K(z_n,s) s^\frac{\gamma+1}{m} v(s) ds 
	= w_{n,0}(h) v(0)+\sum_{i=1}^{\frac{n-1}{2}} w_{n,2i-1}(h) v(z_{2i-1}) + w_{n,2i}(h) v(z_{2i}) \\
	&+ \frac{1}{2}v''(\xi_1)\int_0^{z_1} K(z_n,s) s^{1+\frac{\gamma+1}{m}}(s-z_1) ds + \sum_{i=1}^{\frac{n-1}{2}} \int_{z_{2i-1}}^{z_{2i+1}} K(z_n,s) s^{\frac{\gamma+1}{m}}(s-z_{2i-1})(s-z_{2i}) v''(\xi_i(s)) ds,
\end{split}
\label{eqn:SecondOrderOdd}
\end{equation}
where $\xi_s$ is an intermediate point. In that case, the weights are given by
\begin{equation}
	\begin{split}
	w_{n,0}(h) = \int_{0}^{z_{1}} K(z_n,s) s^{\frac{\gamma+1}{m}}\left(1-\frac{s}{h}\right)ds, \\
	w_{n,1}(h) = \int_{0}^{z_{1}} K(z_n,s) s^{\frac{\gamma+1}{m}}\frac{s}{h}ds + \int_{z_1}^{z_3} K(z_n,s) s^{\frac{\gamma+1}{m}}\left(1-\frac{s-z_{1}}{h}\right)ds, \\
	w_{n,i}(h) = \int_{z_{i-1}}^{z_{i+1}} K(z_n,s) s^{\frac{\gamma+1}{m}}
	\left( 
	\begin{cases}
		\frac{s-z_{i-1}}{h}, & i \text{ even} \\
		1-\frac{s-z_{i-1}}{h}, & i \text{ odd} \\
	\end{cases}
	\right) ds
	\quad \text{for } n \text{ odd and } 1< i\leq n-1.
	\end{split}
\label{eqn:WeightsOdd}
\end{equation}
Therefore, the whole equation (\ref{eqn:VolterraEqV}) using our linear reconstruction can be written as
\begin{equation}
	v(z_n)^{m+1} = z_n^{-\frac{(m+1)(\gamma+1)}{m}} \sum_{i=0}^{n-1} w_{n,i}(h) v(z_i) + \delta_n(h),
\label{eqn:SecondOrderExact}
\end{equation}
where the remainder can be estimated with the help of (\ref{eqn:SecondOrderEven}) and (\ref{eqn:SecondOrderOdd}) to be
\begin{equation}
	|\delta_n(h)| \leq K_+\|v''\| B\left(\gamma+1, \frac{\gamma+1}{m}\right) h^2 =: C h^2 = O(h^2) \quad \text{as} \quad h\rightarrow 0^+,
\label{eqn:SecondOrderRemainder}
\end{equation}
uniformly for $n\in\mathbb{N}$. Therefore, by truncating the remainder we obtain a consistent second order method in the form (\ref{eqn:NumMet}) where weights are defined in (\ref{eqn:WeightsEven}) and (\ref{eqn:WeightsOdd}). 

Having designed a method we can move to the important question about the initial conditions for $v$. Since (\ref{eqn:VolterraY}) always possesses a trivial solution $y\equiv 0$ we have to choose an appropriate starting value for our numerical method in order to converge to a non-trivial one. One way of doing that is to approximate the integral in (\ref{eqn:VolterraEqV}) over the interval $z\in[0,h]$ with some simple quadrature. For example, we can use the rectangle rule in which we take the value of $v$ at the \emph{right} endpoint. That is, we reconstruct $v$ with a constant function
\begin{equation}
	v(s) = v(h) - v'(\xi_h)(h-s), \quad s\in(0,h),
\end{equation} 
for some $\xi_h \in (0,h)$. Then, from (\ref{eqn:VolterraEqV}) we can write
\begin{equation}
	v(h)^{m+1} = h^{-\frac{(m+1)(\gamma+1)}{m}} v(h)\int_0^h K(h,s) s^\frac{\gamma+1}{m} ds + R(h),
\label{eqn:InitialStep0}
\end{equation}
where by the mean value theorem the remainder satisfies
\begin{equation}
	|R(h)| = h^{-\frac{(m+1)(\gamma+1)}{m}} |v'(\widehat{\xi}_h)|\int_0^h K(h,s) (h-s) s^\frac{\gamma+1}{m} ds \leq K_+ |v'(\widehat{\xi}_h)| B\left(\gamma+2, \frac{\gamma+1}{m}+1\right) h,
\end{equation}
where we have used the assumption of kernel boundedness (\ref{eqn:KernelBounds}). Therefore, after substitution $s=z \sigma$ in (\ref{eqn:InitialStep0}) we have
\begin{equation}
	v(h)^{m+1} -v(h) h^{-\gamma} \int_0^1 K(h,h\sigma) \sigma^\frac{\gamma+1}{m} d\sigma= R(h).
\end{equation}
Now, the left-hand side of the above is bounded from (\ref{eqn:KernelBounds}) while the right-hand side vanishes to zero when $h\rightarrow 0^+$. We can then propose that the starting step for the numerical method is
\begin{equation}
	v_0 = v(0) = \lim\limits_{h\rightarrow 0^+}\left(h^{-\gamma} \int_0^1 K(h,h\sigma) \sigma^\frac{\gamma+1}{m} d\sigma\right)^\frac{1}{m},
\label{eqn:InitialStep}
\end{equation}
provided the above limit exists. We thus assume that it is the case since otherwise, the solution could be unnecessarily difficult or even meaningless to solve numerically. Note that, taking the above as an initial step in the scheme does not introduce any error apart from numerical rounding. In practice we can use for example the Gaussian quadrature to compute the integral above with finite $h$ fixed to be the numerical method step. 

The initial step (\ref{eqn:InitialStep}) is sufficient to start the rectangle scheme (\ref{eqn:Rectangle}) however, in order to initialize the second order trapezoidal method we need a guess on the value of $v(h)$. A straightforward idea is to once again use the linear interpolation between $v(0)$ and $v(h)$ similarly as in the odd case of method's construction. But in present case it yields $v(h)$ implicitly
\begin{equation}
\begin{split}
	v(h)^{m+1} &= h^{-\frac{(m+1)(\gamma+1)}{m}} \int_0^h K(h,s) s^\frac{\gamma+1}{m} v(s) ds \\
	&= h^{-\frac{(m+1)(\gamma+1)}{m}} \left[v(0) \int_0^h K(h,s) s^\frac{\gamma+1}{m} \left(1-\frac{s}{h}\right) ds + v(h) \int_0^h K(h,s) s^\frac{\gamma+1}{m} \frac{s}{h}ds\right] + R(h),
\end{split}
\end{equation} 
where, similarly as before, the remainder is $R(h) = O(h^2)$ as $h\rightarrow 0^+$. After truncation, the second initial condition $v_1$ can be found by solving the nonlinear equation
\begin{equation}
	 h^{\frac{(m+1)(\gamma+1)}{m}} v_1^{m+1} - \left( \int_0^h K(h,s) s^\frac{\gamma+1}{m} \frac{s}{h}ds \right) v_1 - \left(\int_0^h K(h,s) s^\frac{\gamma+1}{m} \left(1-\frac{s}{h}\right) ds \right) v_0 = 0.
\label{eqn:InitialStep1}
\end{equation}
Observe that solving the above is required only once at the initialization phase of the second order method. Similarly, as with the computation of $v_0$ we can use Gaussian quadrature to approximate the above integrals. Moreover, since the the above nonlinear equation is of the form $a x^{m+1} - b x - c = 0$ with a known derivative, we can readily use Newton's method to solve it. Therefore, (\ref{eqn:InitialStep}) and (\ref{eqn:InitialStep1}) supply us with at worst a second order approximation of the starting values for the scheme (\ref{eqn:NumMet}) along with weights (\ref{eqn:WeightsEven}) and (\ref{eqn:WeightsOdd}). 

\subsection{Convergence}
Before we proceed to the convergence proof of (\ref{eqn:NumMet}) with weights (\ref{eqn:WeightsEven}) and (\ref{eqn:WeightsOdd}) we present two auxiliary results that will be used later. The first one is a variation on a discrete version of the classical Gr\"onwall-Bellman's lemma. These lemmas are essential in studying convergence of numerical methods for partial differential equations in both the classical and fractional setting. In the literature one can find many different versions of them tailored for specific needs. For example, they are an essential tool in the important case of subdiffusion \cite{Lia19} when they are usually used to solve the error inequality. In the cited work authors prove a general case of the discrete fractional Gr\"onwall-Bellman's inequality where they allow for general weights present in the sum. This generality helps to tackle convergence of a variety of numerical schemes based on, for example, nonuniform L1 method or Alikhanov's discretisations. Other interesting results can be found in \cite{Li18, McK82}.

\begin{lem}\label{lem:GB}
Let $\left\{e_n\right\}$, $n=1,2,...$ be a sequence of positive numbers satisfying 
\begin{equation}
	e_n \leq \frac{\mu}{n} \sum_{i=1}^{n-1} e_i + \delta, \quad n\geq 1,
	\label{eqn:LemRecur}
\end{equation}
where $\mu>0$ and $\delta>0$. Then, we have
\begin{equation}
	e_n \leq \delta f_n,
\end{equation}
where 
\begin{equation}
	f_n = \frac{\Gamma(n+\mu)}{n!} \sum_{k=0}^{n-1} \frac{k!}{\Gamma(k+1+\mu)}.
\label{eqn:LemmaBound}
\end{equation}
Moreover, 
\begin{equation}
	f_n \begin{cases}
		\leq \dfrac{1}{1-\mu}, & 0<\mu<1, \\
		\leq \ln n + 1, & \mu = 1, \\
		\stackrel{n\rightarrow\infty}{\sim} \dfrac{1}{\Gamma(\mu+1)}\dfrac{n^{\mu-1}}{\mu-1}, & \mu > 1. \\
	\end{cases}
\label{eqn:LemmaBoundF}
\end{equation}
\end{lem}
\begin{proof}
We proceed by mathematical induction. To begin, let us observe that $e_1 \leq \delta$ by the convention that $\sum_{i=1}^{0} = 0$. Hence, $f_1 = 1$. Suppose that $e_i \leq \delta f_i$ for $1\leq i \leq n-1$. Due to this inductive assumption we have
\begin{equation}
	e_n \leq \delta \left(\frac{\mu}{n}\sum_{i=1}^{n-1} f_i + 1\right).
\label{eqn:Induction}
\end{equation}
We claim that the sequence $f_i$ defined in (\ref{eqn:LemmaBound}) satisfies the following difference equation
\begin{equation}
	\frac{\mu}{n}\sum_{i=1}^{n-1} f_i + 1 = f_n,
\label{eqn:DifferenceEquation}
\end{equation}
which combined with (\ref{eqn:Induction}) proves the inductive assertion. When we subtract (\ref{eqn:DifferenceEquation}) written for $n+1$ the equation for $f_n$ we can obtain a local equation
\begin{equation}
	(n+1)f_{n+1} - (n+\mu) f_n - 1 = 0.
\end{equation} 
Switching back to $f_n$ and rearranging we arrive at the recurrence
\begin{equation}
	f_n = \left(1+\frac{\mu-1}{n}\right) f_{n-1} + \frac{1}{n}.
\end{equation} 
The above can be iterated to yield
\begin{equation}
	f_n = \prod_{i=2}^n \left(1+\frac{\mu-1}{i}\right) + \sum_{i=0}^{n-2} \frac{1}{n-2} \prod_{j=1}^i \left(1+\frac{\mu-1}{n-j+1}\right).
\end{equation}
Further, the products can be simplified considerably when we notice that
\begin{equation}
	\prod_{i=2}^n \left(1+\frac{\mu-1}{i}\right) = \frac{(1+\mu)(2+\mu)...(n+\mu-1)}{n!} = \frac{1}{n!}\frac{\Gamma(n+\mu)}{\Gamma(\mu+1)},
\end{equation}
and similarly for the second one. Consequently, we arrive at
\begin{equation}
\begin{split}
	f_n &= \frac{\Gamma(n+\mu)}{n!}\left(\frac{1}{\Gamma(\mu+1)} + \sum_{k=2}^{n}\frac{1}{k} \frac{k!}{\Gamma(k+\mu)}\right) = \frac{\Gamma(n+\mu)}{n!}\left(\frac{1}{\Gamma(\mu+1)} + \sum_{k=1}^{n-1}\frac{k!}{\Gamma(k+1+\mu)}\right) \\ 
	&=  \frac{\Gamma(n+\mu)}{n!}\sum_{k=0}^{n-1} \frac{k!}{\Gamma(k+1+\mu)},
\end{split}
\end{equation}
which is (\ref{eqn:LemmaBound}) and solves the recurrence. 

In order to show (\ref{eqn:LemmaBoundF}) it suffices to consider the recurrence (\ref{eqn:DifferenceEquation}). Suppose that $0<\mu<1$, then $f_1 = 1 < 1/(1-\mu)$. Further, if $f_i \leq 1/(\mu-1)$ then
\begin{equation}
	f_n \leq \frac{\mu}{n} \sum_{i=1}^{n-1} \frac{1}{1-\mu} + 1 \leq \frac{\mu}{1-\mu} + 1 = \frac{1}{1-\mu}.
\end{equation}
Next, for $\mu = 1$ we trivially have $f_1 = 1 = \ln 1 + 1$, and further by inductive assumption
\begin{equation}
	f_n \leq \frac{1}{n} \sum_{i=1}^{n-1} \left(\ln i + 1\right) + 1 \leq \frac{1}{n} \int_1^n \left(\ln x + 1\right) dx + 1,
\end{equation}
where we have used the fact that the sum is bounded by an appropriate integral. Evaluating, we obtain
\begin{equation}
	f_n \leq \frac{1}{n} \left.x \ln x \right|_1^n + 1 = \ln n + 1.
\end{equation}
Next, for the case $\mu > 1$ it is convenient to use the exact formula for $f_n$, that is (\ref{eqn:LemmaBound}). Then
\begin{equation}
	f_n \sim n^{\mu-1} \sum_{k=0}^{\infty} \frac{k!}{\Gamma(k+1+\mu)} \quad \text{as} \quad n\rightarrow\infty,
\end{equation}
where we have used Stirling's formula for the prefactor and to ascertain the series convergence: $k!/\Gamma(k+1+\mu)\sim k^{\mu}$ as $k\rightarrow\infty$. We can find the exact form of the above sum by using the relation between beta and gamma functions
\begin{equation}
	\sum_{k=0}^{\infty} \frac{k!}{\Gamma(k+1+\mu)} = \frac{1}{\Gamma(\mu)} \sum_{k=0}^{\infty} \frac{\Gamma(k+1)\Gamma(\mu)}{\Gamma(k+1+\mu)} = \frac{1}{\Gamma(\mu)} \sum_{k=0}^{\infty} \int_0^1 (1-x)^{\mu-1} x^k dx,
\end{equation}
and by Tonelli's theorem we can exchange the order of summation and integration
\begin{equation}
	\sum_{k=0}^{\infty} \frac{k!}{\Gamma(k+1+\mu)} = \frac{1}{\Gamma(\mu)} \int_0^1 (1-x)^{\mu-1} \sum_{k=0}^{\infty} x^k dx = \frac{1}{\Gamma(\mu)} \int_0^1 (1-x)^{\mu-2} dx = \frac{1}{(\mu-1)\Gamma(\mu)},
\end{equation}
which proves the last assertion. 
\end{proof}
For a thorough presentation of similar results see \cite{Ame97}. Next, we prove boundedness of the numerical approximation. 

\begin{lem}\label{lem:BoundV}
Let $v_n$ be the numerical approximation of $v(z_n)$ calculated from (\ref{eqn:NumMet}). Fix sufficiently small number $\epsilon > 0$ and choose $h>0$ small enough for $(z_n)^{-(m+1)(\gamma+1)/m}\sum_{i=0}^{n-1}|\delta_i(h)|<\epsilon$. Then, for $n\geq 1$ we have
\begin{equation}
	0< \left(K_- \left(\beta\left(\gamma+1,\frac{\gamma+1}{m}+1\right)-\epsilon\right)\right)^\frac{1}{m} \leq v_n \leq \left(K_- \left(\beta\left(\gamma+1,\frac{\gamma+1}{m}+1\right)+\epsilon\right)\right)^\frac{1}{m},
\label{eqn:BoundsVn}
\end{equation}
provided that the initial step $v_0$ satisfies the same inequality. Here, $\beta(\cdot,\cdot)$ denotes the beta function.
\end{lem}
\begin{proof}
We will proceed by induction. The first step is satisfied by the assumption. Suppose further that $v_i$ satisfies (\ref{eqn:BoundsVn}) for all $0\leq i<n-1$. We will show that the same is true for $v_n$. To this end, write $V_-^{1/m} \leq v_i \leq V_+^{1/m}$ where $V_\pm$ are defined in (\ref{eqn:BoundsVn}). By (\ref{eqn:NumMet}) and (\ref{eqn:KernelBounds}) we have
\begin{equation}
	v_n^{m+1} \geq V_-^\frac{1}{m} K_- z_n^{-\frac{(m+1)(\gamma+1)}{m}} h \sum_{i=0}^{n-1} w_{n,i} (z_n-z_i)^\gamma z_i^\frac{\gamma+1}{m}.
\end{equation}
Now, according to (\ref{eqn:Quadrature}) we have
\begin{equation}
	h \sum_{i=0}^{n-1} w_{n,i} (z_n-z_i)^\gamma z_i^\frac{\gamma+1}{m} = \int_0^{z_n} (z_n-s)^\gamma s^\frac{\gamma+1}{m} ds = z_n^{\frac{(\gamma+1)(m+1)}{m}} \beta\left(\gamma+1,\frac{\gamma+1}{m}+1\right) - \sum_{i=0}^{n_1}\delta_i(h).
\end{equation}
Combining the two above expressions we obtain
\begin{equation}
	v_n \geq V_-^\frac{1}{m(m+1)} \left(K_- \left( \beta\left(\gamma+1,\frac{\gamma+1}{m}+1\right) - z_n^{-\frac{(m+1)(\gamma+1)}{m}} \delta_n(h) \right)\right)^\frac{1}{m+1}.
\end{equation}
Now, since by the assumption the remainder is smaller than $\epsilon$ we can write
\begin{equation}
	v_n \geq V_-^\frac{1}{m(m+1)} \left(K_- \left( B\left(\gamma+1,\frac{\gamma+1}{m}+1\right) - \epsilon \right)\right)^\frac{1}{m+1} = V_-^\frac{1}{m(m+1)} V_-^\frac{1}{m+1} = V_-^\frac{1}{m},
\end{equation}
what completes the induction. The proof of the upper bound proceeds identically. 
\end{proof}
As it will be seen in the following main results, the boundedness from below plays a crucial role in the proof. 

A remark concerning the nature of the nonlinearity in (\ref{eqn:VolterraY}) can be revealing. Note that if we substitute $f = y^{m+1}$ then the equation transforms into
\begin{equation}
	f(z) = \int_0^z K(z,s) f(s)^\frac{1}{m+1} ds,
\end{equation}
in which the nonlinearity is manifestly nonlipschitzian. Therefore, we cannot use the general theory to conclude the convergence. To wit, for the Lipschitz nonlinearity we are always able to reduce the analysis of the method's error $e_n$ to investigations of the following recurrence inequality (see \cite{Lin85})
\begin{equation}
\label{eqn:ErrorLipschitz}
	\left|e_n\right| \leq \mu h \sum_{i=1}^{n-1} \left|e_i\right| + \delta(h),
\end{equation}
where $\mu$ and $\nu=\nu(h)$ are independent of $n$. Using the well-known discrete version of the Gr\"onwall-Bellman's lemma we can solve the above to yield (compare with our Lemma \ref{lem:GB})
\begin{equation}
	\left|e_n\right| \leq \delta(h) e^{\mu nh}.
\end{equation}
The term $\delta(h)$ depends on the local consistency error and vanishes for $h\rightarrow 0$. Moreover, since we are comparing the numerical and exact solutions at a fixed point, we have $nh\rightarrow \text{const.}$. Therefore, $e_n\rightarrow 0$ and the method is convergent. The lipschitzian character of the nonlinearity makes the proof similar as in the linear equations. For the non-Lipschitz case we cannot in general write (\ref{eqn:ErrorLipschitz}) and thus the arguments have to be different and more subtle. Notice that this difficulty is the very consequence of the degeneracy of the main equation (\ref{eqn:DiffEqPDE}). We summarize the result in the following theorem. 

\begin{thm}\label{thm:Covergence}
Fix $z\in(0,1]$ and chose the weights of the quadrature (\ref{eqn:Quadrature}) according to (\ref{eqn:WeightsEven}) and (\ref{eqn:WeightsOdd}) while the starting values from (\ref{eqn:InitialStep}) and (\ref{eqn:InitialStep1}). Then, if  
\begin{equation}
	\mu_m := \frac{4 K_+}{(m+1) V_-} < 3,
\label{eqn:mum}
\end{equation}
the scheme (\ref{eqn:NumMet}) is convergent to the nontrivial solution of (\ref{eqn:VolterraEqV}) with an order at least equal to 
\begin{equation}
	\min\{2,3-\mu_m\}.
\label{eqn:Order}
\end{equation}
Here, $V_-^{1/m}$ is the lower bound for both $v=v(z)$ and $v_n$. 
\end{thm}
\begin{proof}
Let us define the error of the numerical approximation by $e_n = v(z_n) - v_n$, where $v(z)$ is the solution of (\ref{eqn:VolterraEqV}) while $v_n$ comes from (\ref{eqn:NumMet}). Now, the difference between these two equations is
\begin{equation}
	v(z_{n})^{m+1}-v_n^{m+1} = z_n^{-\frac{(m+1)(\gamma+1)}{m}} \left(\int_0^z K(z,s) s^\frac{\gamma+1}{m} v(s) ds-\sum_{i=0}^{n-1} w_{n,i}(h) v_i\right).
\end{equation}
With the use of (\ref{eqn:Quadrature}) and the construction leading to (\ref{eqn:SecondOrderExact}) we further obtain
\begin{equation}
	v(z_{n})^{m+1}-v_n^{m+1} = z_n^{-\frac{(m+1)(\gamma+1)}{m}} \sum_{i=1}^{n-1} w_{n,i}(h) e_i + \delta_n(h),
\end{equation}
where the remainder $\delta_n(h)$ satisfies (\ref{eqn:SecondOrderRemainder}) and the zero term vanishes due to exact starting value (\ref{eqn:InitialStep}). On the other hand, by the mean value theorem we can write
\begin{equation}
	v(z_{n})^{m+1}-v_n^{m+1} = (m+1)\xi_n^m e_n,
\end{equation}
where $\xi_n$ lies between $v(z_n)$ and $v_n$. Next, by combining the two above equations we have
\begin{equation}
	(m+1)\xi_n^m |e_n| \leq z_n^{-\frac{(m+1)(\gamma+1)}{m}} \sum_{i=1}^{n-1} |w_{n,i}(h)| |e_i| + C h^2,
\end{equation}
where $C$ is a constant explicitly defined in (\ref{eqn:SecondOrderRemainder}). The boundedness of both $v(z_n)$ and $v_n$ (see (\ref{eqn:BoundsV}) and Lemma \ref{lem:BoundV})) implies that $\xi_n \geq V_-$ for some constant and thus
\begin{equation}
	|e_n|\leq \frac{z_n^{-\frac{(m+1)(\gamma+1)}{m}} }{(m+1)V_-} \sum_{i=1}^{n-1} |w_{n,i}(h)| |e_i| + \frac{C h^2}{(m+1)V_-},
\end{equation}
which is a recurrence inequality which will eventually be solved with invoking Gr\"onwall-Bellman's lemma. To see this observe that by (\ref{eqn:WeightsEven}) and (\ref{eqn:WeightsOdd}) we have for example
\begin{equation}
	|w_{n,i}(h)| \leq \frac{4 K_+}{n} z_n^{\frac{(m+1)(\gamma+1)}{m}},
\end{equation}
where the factor $4/n$ comes from the fact that $|s-z_i|\leq 2h$ on each subinterval of length $2/n$, $K_+$ is the bound on the kernel (\ref{eqn:KernelBounds}), and the integrand is not larger than $1$. This leads to 
\begin{equation}
	|e_n|\leq \frac{4 K_+}{(m+1)V_-} \frac{1}{n} \sum_{i=1}^{n-1} |e_i| + \frac{C h^2}{(m+1)V_-} =: \frac{\mu_m}{n} \sum_{i=1}^{n-1} |e_i| + \delta.
\label{eqn:NonlocalRecurrence}
\end{equation}
Invoking Lemma \ref{lem:GB} thus yields
\begin{equation}
	|e_n| \leq \frac{C h^2}{(m+1)V_-} f_n,
\end{equation}
where $f_n$ is defined in (\ref{eqn:LemmaBound}). Now, if $0<\mu_m<1$, the error $|e_n|$ is bounded for all $n$, i.e.
\begin{equation}
	|e_n| \leq \frac{C h^2}{(m+1)V_--4 K_+}
\end{equation}
Further, for $\mu_m > 1$ we have
\begin{equation}
	|e_n|\leq \frac{C h^2}{(m+1)V_-} f_n \sim \frac{C h^2}{(\mu_m-1)\Gamma(\mu_m)(m+1)V_-} n^{\mu_m-1} \sim \frac{C z^{\mu_m-1}}{(\mu_m-1)\Gamma(\mu_m)(m+1)V_-} h^{3-\mu_m},
\end{equation} 
as $n\rightarrow\infty$. Here, we have used the fact that $nh \rightarrow z$. Similarly, the marginal case $\mu=1$ yields a convergence of order $-h^2 \ln h$. The proof is complete.  
\end{proof}
According to (\ref{eqn:Order}) the method retains its order for sufficiently small $\mu_m$. More specifically, when $\mu_m < 1$ the method is second order accurate. This estimate gradually becomes worse until $\mu_m = 3$ when the theorem does not guarantee convergence. Note, however, that choosing sufficiently large $m$ we always can obtain a second order method. The fact that, according to the theorem, the order of the method can be smaller than the order of the quadrature is probably the consequence of severe nonlinearity of the equation. This, in turn, corresponds to the degeneracy of the original PDE (\ref{eqn:DiffEqPDE}) what throughout the years has proved to lay significant difficulties for both theoretical and numerical studies (see for ex. \cite{Dib84,Eti12}). The main difference between the classical case of Lipschitzian nonlinearity is the occurrence of $1/n$ instead of $h$ in the nonlocal recurrence (\ref{eqn:NonlocalRecurrence}). This forces us to use Lemma \ref{lem:GB}) rather than the classical discrete Gr\"onwall-Bellman's lemma. Since $1/n$ changes in each recurrence step the assertion is somewhat different yielding a loss in convergence order. However, for $0<\mu_m<1$ the error is $O(h^2)$ as $h\rightarrow 0^+$ for any $n > 0$. A very thorough account of numerical methods for integral equations with Lipschitzian nonlinearities can be found in \cite{Lin85}. 

\section{Numerical examples}\label{sec:NumericalExamples}
In this section we present several numerical examples. We start with a simple integral equation that can be solved exactly in a closed form. Next, we consider a more complex equation and use it to verify the order of convergence. Further, we proceed to solving the time-fractional diffusion problems summarized in Tab. \ref{tab:AB} and we end this section with temporal computational complexity estimates and comparison with finite difference scheme. Numerical calculations have been conducted in Julia programming language. Integrals appearing in weights (\ref{eqn:WeightsEven}) and (\ref{eqn:WeightsOdd}) have been calculated with the QuadGK package that utilizes Gauss-Kronrod adaptive quadrature.

\subsection{Exact constant solution}
We start with a simple illustration of the exactness of our numerical methods. Let $K(z,s) = (z-s)^\gamma$ in (\ref{eqn:VolterraEqV}). According to (\ref{eqn:ConstantSolution}) the integral equation has then a constant solution (that is, a power function is a solution of (\ref{eqn:VolterraY})). In Tab. \ref{tab:ExactSolutionError} we have collected the maximal absolute error (first norm) of the numerical approximation (\ref{eqn:NumMet}) with second order weights (\ref{eqn:WeightsEven}) and (\ref{eqn:WeightsOdd}). Note that in calculations we have used $N=10$ which is relatively small. Our experiments showed that increasing the number of interval divisions does not improve the error. Notice also that the smallest error $O(10^{-16})$ appears for the case $\gamma=0$ and $m=1$ for which the exact solution is equal to $1/2$. The majority of simulations concluded that the error is of order of $O(10^{-11})$ we can conclude that the method performs very well and reproduces the exact solution with good accuracy. 

\begin{table}
	\centering
	\begin{tabular}{rllll}
		\toprule
		& $\gamma = 0$ & $\gamma = 0.5$ & $\gamma = \sqrt{2}$ & $\gamma = \pi$ \\
		\midrule
		$m = 1$ & $1.11\times 10^{-16}$ & $2.90\times 10^{-11}$ & $8.94\times 10^{-13}$ & $2.20\times 10^{-14}$ \\
		$m = 2$ & $2.40\times 10^{-10}$ & $1.18\times 10^{-10}$ & $1.19\times 10^{-11}$ & $1.60\times 10^{-13}$ \\
		$m = 10$ & $1.90\times 10^{-10}$ & $1.01\times 10^{-10}$ & $4.26\times 10^{-11}$ & $4.71\times 10^{-11}$ \\
		$m = 100$ & $1.85\times 10^{-11}$ & $2.03\times 10^{-11}$ & $1.97\times 10^{-11}$ & $1.01\times 10^{-11}$ \\
		\bottomrule
	\end{tabular}
	\caption{Maximal absolute error for a numerical approximation of the solution of (\ref{eqn:VolterraEqV}) with $K(z,s) = (z-s)^\gamma$ for various $\gamma$ and $m$. The number of interval divisions is $N=10$.}
	\label{tab:ExactSolutionError}
\end{table}

\subsection{Order of convergence}
We will calculate the empirical order of convergence of our second order method and compare it with the rectangle scheme (\ref{eqn:Rectangle}). This illustration will be completed with the use of the integral equation (\ref{eqn:VolterraEqV}) with 
\begin{equation}
	K(z,s) = \frac{\sqrt{z-s}}{1+\sin^2 s}.
\end{equation}
Obviously, the kernel satisfies our assumptions on the boundedness (\ref{eqn:KernelBounds}) with $K_- = 1/2$ and $K_+ = 1$. The order of convergence is estimated with the extrapolation technique known as Aitken's method (see \cite{Lin85})
\begin{equation}
	\text{order} \approx \log_2 \frac{\left|v^{(2N)}_{2N}-v_N^{(N)}\right|}{\left|v^{(4N)}_{4N}-v_{2N}^{(2N)}\right|},
\end{equation}
which compares the numerical solution evaluated at $z=1$ (the worst case) with calculations for $N$, $2N$, and $4N$ steps. The summary of obtained results is presented in Tab. \ref{tab:Order}. We can see that the estimated order is near 2 however, for larger $m$ it is somewhat lower. This fact probably originates in the necessity of taking a very high order root in the equation. Moreover, for this example, we have $\mu_m = 0.8 < 3$ what satisfies the assumption of Theorem \ref{thm:Covergence}. We can thus see that, as anticipated, the method is of second order. Moreover, similar calculations have been done for the rectangle method (\ref{eqn:Rectangle}) and the results were more uniform: the estimated order was equal to $0.99$ for all considered $m$. 

\begin{table}
	\centering

	\begin{tabular}{cccccccc}
		\toprule
		$m$ & 1 & 5 & 10 & 15 & 20 & 50 & 100 \\
		\midrule
		order & 2.02 & 2.00 & 1.96 & 1.92 & 1.90 & 1.86 & 1.84 \\
		\bottomrule
	\end{tabular}
	\caption{Numerically calculated order of convergence of the numerical method based on trapezoidal approximation (\ref{eqn:WeightsEven}), (\ref{eqn:WeightsOdd}). The base for Aitken's estimation is $N=100$. }
	\label{tab:Order}
\end{table}

\subsection{Anomalous diffusion}
Now, we proceed to the subject of the main interest - the time-fractional porous medium equation. We use the kernel (\ref{eqn:VolterraRobin}) which conveniently can be written with the help of incomplete beta functions
\begin{equation}
\begin{split}
	K(z,u) &= \frac{m+1}{\Gamma(1-\alpha)} \left[B(1-u)\beta(1-\alpha, a+2b+1, 1-w) \right. \\
	&\left.+ (A+B) \left((1-u) \beta(1-\alpha, a+2 b+1, 1-w) - (1-z) \beta(1-\alpha, a+b+1, 1-w)\right)\right],
\end{split}
\end{equation}
where
\begin{equation}
	\beta(a,b,z) = \int_0^z t^{a-1}(1-t)^{b-1} dt,
\end{equation}
and $A$, $B$, $a$, $b$ are chosen according to the boundary conditions summarized in Tab. \ref{tab:AB}. Some exemplary plots of the self-similar solutions of (\ref{eqn:DiffEqPDE}) are presented on Fig. \ref{fig:DiffSS}. 

\begin{figure}
	\centering
	\includegraphics[scale = 0.7]{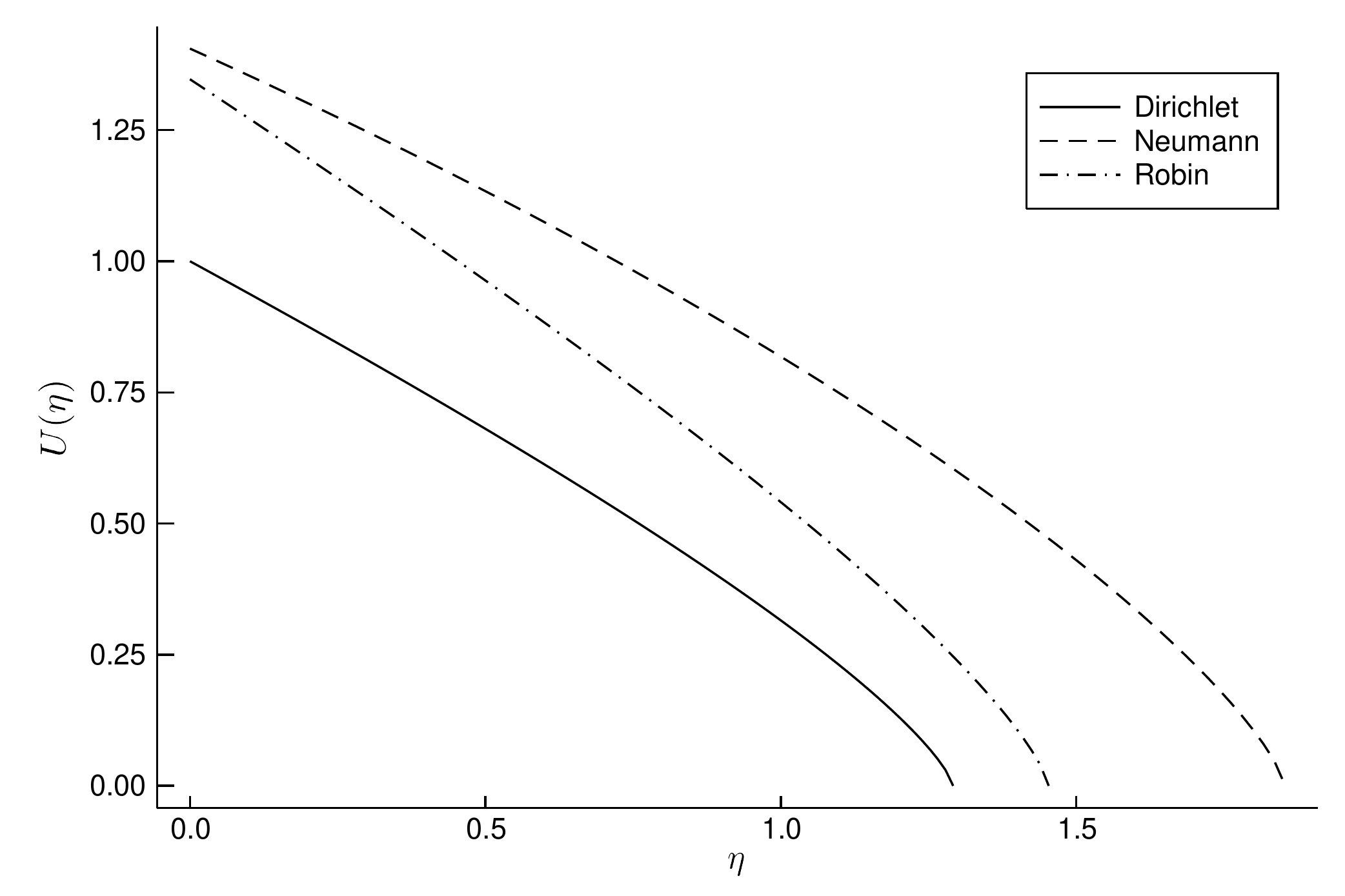}
	\caption{Exemplary self-similar solutions of the time-fractional porous medium equation (\ref{eqn:DiffEqPDE}). Here, $\alpha = 0.5$ and $m=2$. }
	\label{fig:DiffSS}
\end{figure}

First, we estimate the minimal value of parameter $m$, say $m_0$, for which the assumption of the Theorem \ref{thm:Covergence} is satisfied for the most important - Dirichlet - boundary condition. According to (\ref{eqn:mum}) it is equal to a zero of a function $m \mapsto \mu_m - 3$. In order to estimate that we have to find $K_+$ which can be a-priori calculated from (\ref{eqn:K+}). However, a more accurate bound comes from the estimate
\begin{equation}
	K_+ \geq \sup_{0\leq u\leq z \leq 1} \frac{K(z,u)}{(z-u)^{1-\alpha}},
\end{equation} 
which can be found numerically. Note also that when we compare the formulas (\ref{eqn:K+}) and (\ref{eqn:BoundsVn}) we can expect that the ratio $K_+/V_-$ is constant with respect to $m$. Whence $\mu_m$ should be $O(m^{-1})$ as $m\rightarrow\infty$. Results of numerical estimation are presented in Tab. \ref{tab:m0}. We can see that usually the value of $m_0$ is slightly larger than $1$ and closes to $3$ for very small $\alpha$. In applications, one usually finds $m \approx 7-8$ with $\alpha \in [0.6, 1]$ \cite{Sun13} and, for these important cases, we know that Theorem \ref{thm:Covergence} guarantees convergence. 

\begin{table}
	\centering
	\begin{tabular}{rccccccccccc}
		\toprule
		$\alpha$ & 0.99 & 0.9 & 0.8 & 0.7 & 0.6 & 0.5 & 0.4 & 0.3 & 0.2 & 0.1 & 0.01 \\
		\midrule
		$m_0$ & 1.10 & 1.14 & 1.19 & 1.25 & 1.28 & 1.35 & 1.66 & 1.98 & 2.31 & 2.62 & 2.84 \\
		\bottomrule
	\end{tabular}
	\caption{Estimated critical values of $m$ for which the assumptions of Theorem \ref{thm:Covergence} are satisfied for the case of anomalous diffusion.}
	\label{tab:m0}
\end{table}

Similarly as above we can estimate the order of convergence of the trapezoid method. We again use Aitken's algorithm and present the results in Tab. \ref{tab:OrderDiff}. As can be seen the estimated order stays near $2$ especially for moderate values of $m$ what has also been observed in the previous example.  

\begin{table}
	\centering
	\begin{tabular}{rccccccc}
		\toprule
		\backslashbox{$m$}{$\alpha$} & 0.1 & 0.3 & 0.5 & 0.7 & 0.9 & 0.99 \\ 
		\midrule
		1 & 1.99 & 2.00 & 2.00 & 1.98 & 2.08 & 2.07 \\
		3 & 1.96 & 1.99 & 1.99 & 1.98 & 1.97 & 1.96 \\
		5 & 1.93 & 1.96 & 1.96 & 1.95 & 2.01 & 1.92 \\
		7 & 1.90 & 1.93 & 1.93 & 1.93 & 1.95 & 1.88 \\
		10 & 1.87 & 1.91 & 1.91 & 1.91 & 1.92 & 1.85 \\
		15 & 1.86 & 1.89 & 1.90 & 1.90 & 1.91 & 1.81 \\
		\bottomrule
	\end{tabular}
	\caption{Order of convergence for the trapezoidal method applied to subdiffusion for different $\alpha$ and $m$.}
	\label{tab:OrderDiff}
\end{table}

As a further verification of our method we can compute the wetting front position $\eta^*$ for the classical case $\alpha = 1$ and compare it with results from \cite{Okr93}. In that work, the values of $\eta^*$ were computed with a use of power series and we can treat them as exact (up to 9 decimal places). Our calculations for $m=2$ are summarized in Tab. \ref{tab:WettingError}. Results for other values of $m$ are very similar. Notice that even for a very small number of subdivisions we obtain an error of order $O(10^{-5})$ which is a consequence of the second order accuracy. Further, numerical simulations show that on a log-log scale, the error $|\eta^*-\eta^*_{exact}|$ behaves as a line with tangent $-2$. Therefore, we can expect that the $|\eta^*-\eta^*_{exact}| = O(N^{-2})$ as $N\rightarrow\infty$. 

\begin{table}
	\centering
	\begin{tabular}{rccccccc}
		\toprule
		$N$ & 10 & 20 & 50 & 100 & 200 & 500 & 1500 \\
		\midrule
		$|\eta^*-\eta^*_{exact}|$ & $1.1\times 10^{-4}$ & $2.9\times 10^{-5}$ & $4.6\times 10^{-6}$ & $1.1\times 10^{-6}$ & $2.8\times 10^{-7}$ & $4.5\times 10^{-8}$ & $5.4\times 10^{-9}$ \\
		\bottomrule
	\end{tabular}
	\caption{The error in calculating wetting front position for $\alpha = 1$ and $m=2$. The exact value has been taken from \cite{Okr93a}. The error decays at a rate $O(N^{-2})$ as $N\rightarrow\infty$. }
	\label{tab:WettingError}
\end{table}

Wetting fronts for subdiffusive case cannot easily be compared with exact values. However, we have some useful asymptotics. First, let us consider the Dirichlet boundary condition for which we know that (see \cite{Plo19a})
\begin{equation}
	\eta^* = O(m^{-1/2}) \quad m\rightarrow \infty.
\end{equation} 
This relation can also be obtained by combining the value for $\eta^*$ taken from Tab. \ref{tab:AB} with estimates (\ref{eqn:YBounds}) and noting that $K_+ \propto m+1$. Our numerical simulations are depicted on Fig. \ref{fig:Wetting}. Notice that the above asymptotic behaviour is evident even for small values of $m$, i.e. in a log-log scale all lines become parallel to $m^{-1/2}$. 

A similar reasoning can be applied to Neumann and Robin conditions. To this end we need an asymptotic behaviour of the derivative $y'(1)$ for $m\rightarrow\infty$. First, notice that due to (\ref{eqn:K+}) and (\ref{eqn:YBounds}) we have $\|y\| = O((1+m)^{1/m}) = O(1)$ for large $m$. Then, from (\ref{eqn:yy'}) we obtain
\begin{equation}
	\|y'\| = O\left(\frac{1}{m+1}\right) \quad \text{as} \quad m\rightarrow\infty,
\end{equation}
uniformly for $\alpha \in (0,1)$. Next, using the value from Tab. \ref{tab:AB} we can write
\begin{equation}
	\eta^* = \left(\frac{1}{y^m(1)y'(1)}\right)^{\frac{m}{m+2}} = O\left(\frac{m+1}{m+1}\right)^{\frac{m}{m+2}} = O(1) \quad m\rightarrow\infty.
\end{equation}
for Neumann condition. Similarly, the wetting front for the Robin case also is bounded for large $m$. We can see on Fig. \ref{fig:Wetting} that this observation is confirmed with numerical simulations, that is for large $m$ wetting fronts converge to fixed values. 

\begin{figure}
	\centering
	\includegraphics[scale = 0.55]{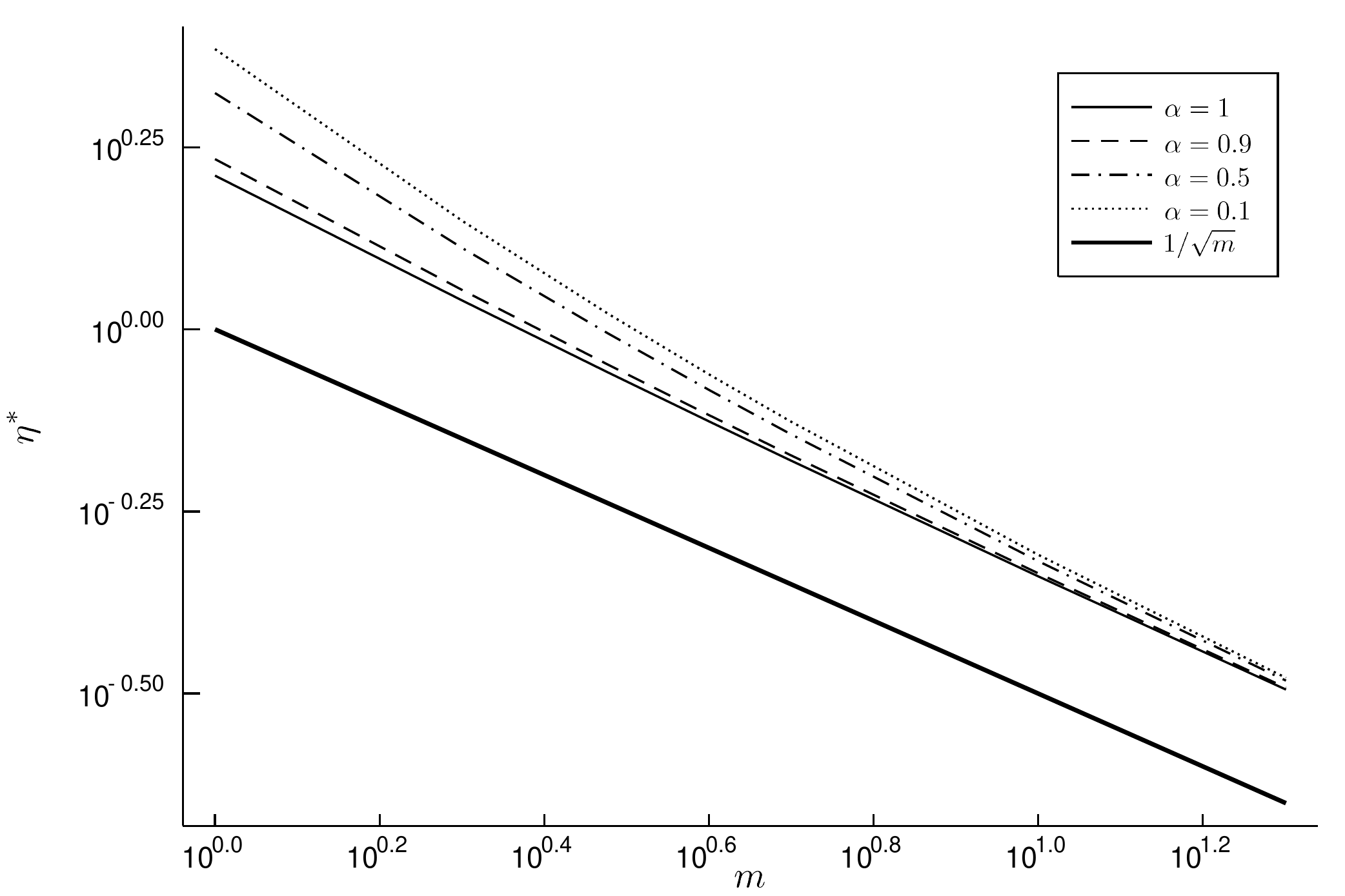}
	\includegraphics[scale = 0.55]{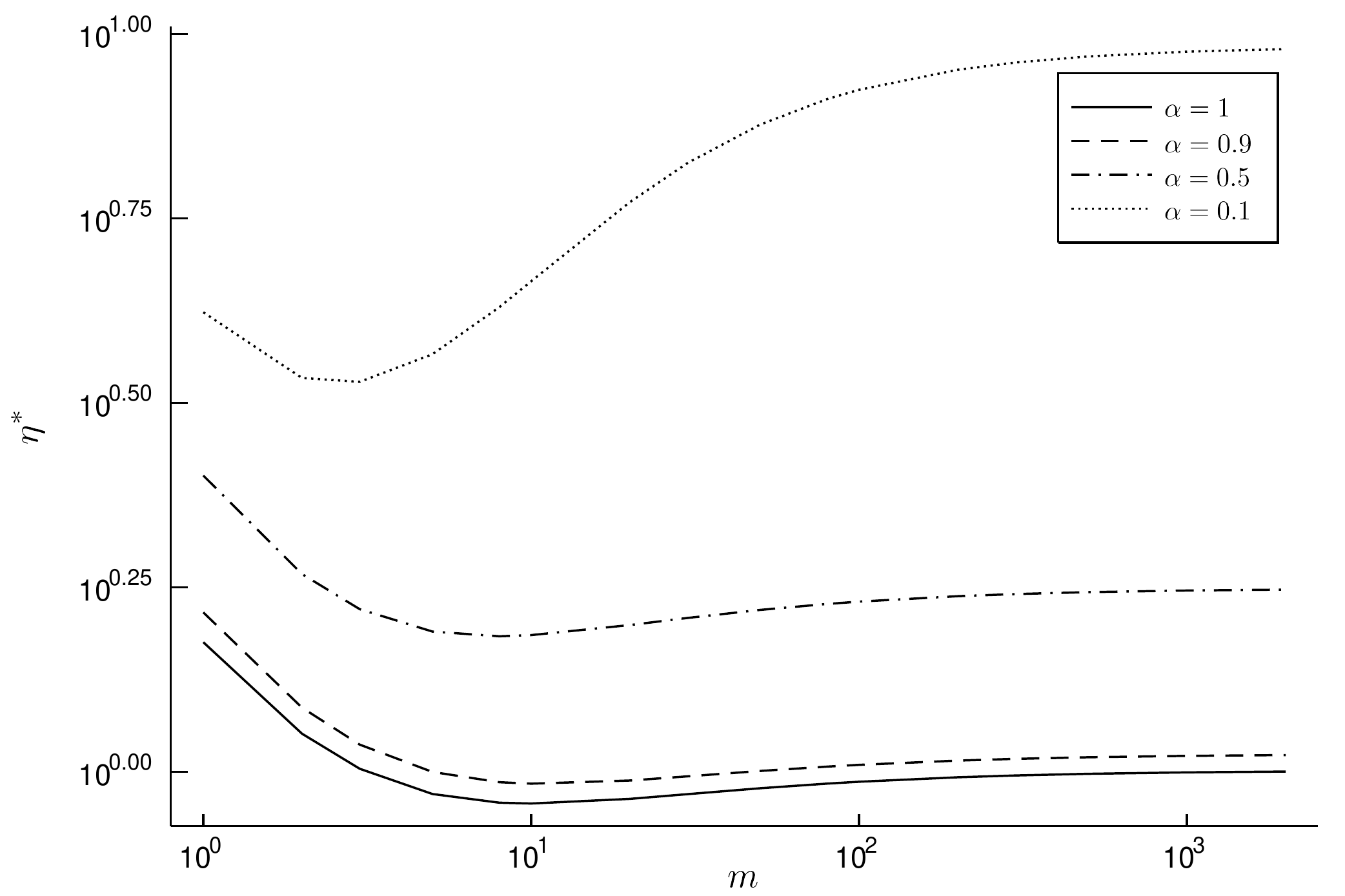}
	\includegraphics[scale = 0.55]{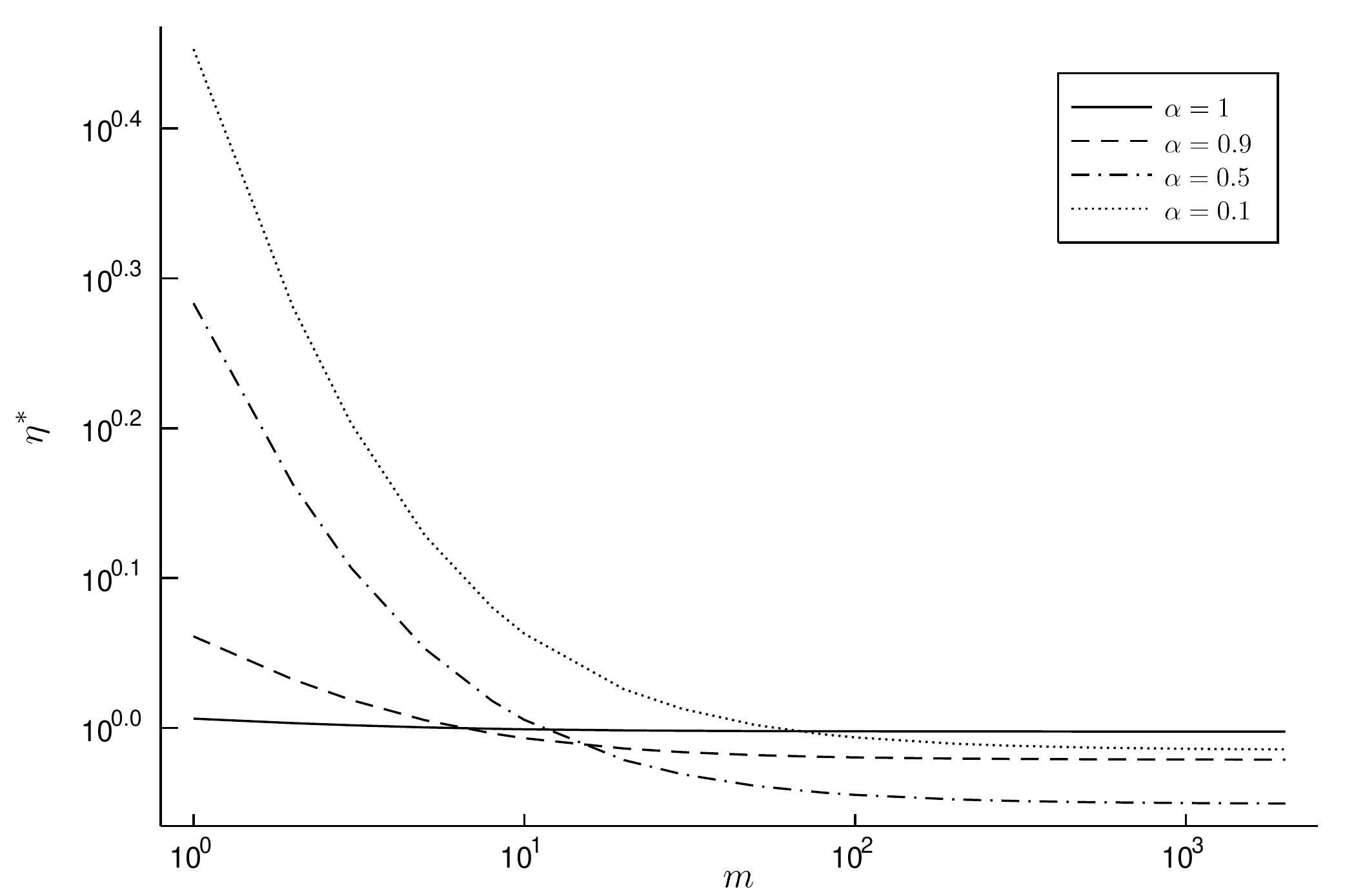}
	\caption{Wetting front position $\eta^*$ with respect to $m$ for a Dirichlet (top), Neumann (middle), and Robin (bottom) boundary conditions shown in a double logarithmic scale. }
	\label{fig:Wetting}
\end{figure}

\subsection{Comparison with finite difference method}
In order to compare our method with some other popular schemes we conduct a computation cost estimates in terms of the temporal complexity. We will contrast the scheme based on the Volterra integral equation (\ref{eqn:NumMet}) with the usual $\theta$-weighted finite difference scheme with L1 discretization of the fractional derivative \cite{Plo14, Li18} (however, the comparison with finite element or spectral methods would yield similar results). The diffusivity is linearized according to \cite{Cra87} and the method for solving (\ref{eqn:DiffEqPDE}) is following
\begin{equation}	
	\begin{split}
		&u_j^{i+1}-\left(1-\theta\right)\frac{h^\alpha}{k^2 \Gamma(2-\alpha)}\left(D^{i+1}_{j-1/2}u^{i+1}_{j-1}+\left(D^{i+1}_{j-1/2}+D^{i+1}_{j+1/2}\right)u^{i+1}_j+D^{i+1}_{j+1/2}u^{i+1}_{j+1}\right) \\
		&=-\sum_{k=1}^i a_{k,i} u^k_j+\theta \frac{h^\alpha}{k^2 \Gamma(2-\alpha)}\left(D^i_{j-1/2}u^i_{j-1}+\left(D^i_{j-1/2}+D^i_{j+1/2}\right)u^i_j+D^i_{j+1/2}u^i_{j+1}\right),
	\end{split}
	\label{eqn:FiniteDifference}
\end{equation}
where the weights are defined by
\begin{equation}
	a_{k,i}=(i-k+2)^{1-\alpha}-2(i-k+1)^{1-\alpha}+(i-k)^{1-\alpha}. 
\end{equation}
and $D^i_{j\pm 1/2}$ is the linearised value of the diffusion coefficient
\begin{equation}
	\begin{split}
		D^i_{j\pm 1/2}= \frac{1}{2}&\left((u^m)^i_j+m (u^{m-1})^i_j\left(u^i_j-u^{i-1}_j\right)+ \right. \\
		& \left. (u^m)^i_{j\pm 1/2}+m (u^{m-1})^i_{j\pm 1/2}\left(u^i_{j\pm 1/2}-u^{i-1}_{j\pm 1/2}\right)\right).
	\end{split}
\end{equation}
Here, $u^i_j$ is the numerical approximation of $u(x_j,t_i)$, where $x_j = j \Delta x$ and $t_i = i \Delta t$. 

Suppose we would like to compute the wetting front position in the Dirichlet problem with a tolerance $\epsilon>0$ (but any other value of $u$ would serve the same purpose for this benchmark). According to Tab. \ref{tab:AB} this requires finding $y(1)=v(1)$. Note that other boundary conditions would need $y'(1)$ which does not change our reasoning. Since the method is of second order, we would like to choose $N>0$ in order to satisfy
\begin{equation}
	|V_N - v(1)| \leq \frac{C}{N^2} \leq \epsilon,
\end{equation}
which immediately gives $N = O(\epsilon^{-1/2})$ as $\epsilon\rightarrow 0^+$. Due to nonlocality of the problem (\ref{eqn:NumMet}), each value $y_n$ arises from all previous ones with $N(N+1)/2$ additions and multiplications with weights $w_{n,i}(h)$. These, in turn have to be calculated by integration for which we assume a constant cost $c_I$. Finally, in each step of the iteration we take a $m+1$-th root. The method is started with initial values (\ref{eqn:InitialStep}) and (\ref{eqn:InitialStep1}) which require $3$ integrations, one root taking, and one solution of the nonlinear algebraic equation which costs $c_E$. Therefore, the total cost of the scheme can be estimated as
\begin{equation}
	\text{temporal complexity of (\ref{eqn:NumMet})} = \left(2+c_I\right)\frac{N(N+1)}{2} + N+1 + 3 c_I + c_E = O(N^2) = O(\epsilon^{-1}),
\end{equation}
as $\epsilon\rightarrow 0^+$. Hence, we obtain a linear relationship between calculating the wetting front and the prescribed tolerance. 

As for the finite difference we assume that we can choose $\theta$ in order for the scheme to be unconditionally stable (this is not trivial since the equation we are solving is nonlinear and degenerate). We are thus allowed to choose $\Delta x = O(\Delta t)$ with $\Delta t = T/N$, where the final time is denoted by $T$. The wetting front $x^*(T)$ is approximated by iterating the scheme up until the final time and then finding $j^*$ such that $u_{j^*}^N >0$ and $u_{j+1}^N = 0$. Since the fractional derivative in our finite differences is discretized using the L1 scheme, the method will have temporal order \textit{at most} equal to $2-\alpha$ (the order depends on the regularity of the solution, see \cite{Kop19}), so that in the best case we have
\begin{equation}
	|u^N_{j^*} - u(x^*(T), T)| = |u^N_{j^*}| \leq C \left(\Delta t^{2-\alpha} + \Delta x^2\right) \leq C \Delta t^{2-\alpha} = C \left(\frac{T}{N}\right)^{2-\alpha} < \epsilon,
\end{equation} 
so that we should have at least $N = O(\epsilon^{-1/(2-\alpha)})$ as $\epsilon \rightarrow 0^+$. Now, similarly as above we can estimate the temporal complexity of the algorithm. Let $c_F$ denote the fixed cost of each time step of the finite difference scheme (\ref{eqn:FiniteDifference}), i.e. it contains all additions and multiplications needed to advance in time. Because of the fractional derivative, in order to arrive at step $N$ we have to evaluate dot products of previous solutions $u^k_j$ with the weights $a_{k,i}$. This gives rise to $2\times N(N+1)/2$ floating point operations. Moreover, since the method is implicit in the stable case, in each time step we have to solve a tridiagonal system which costs $O(N)$ operations. Therefore, we can estimate the total cost of the method
\begin{equation}
	\text{temporal complexity of (\ref{eqn:FiniteDifference})} = N(N+1)c_F + N\times O(N) = O(N^2) = O(\epsilon^{-\frac{2}{2-\alpha}}),
\end{equation}  
as $\epsilon\rightarrow 0^+$. As we can see, the computational cost of calculating the wetting front is always higher for the finite difference case than it is for our method. Note also that we have assumed the best case for the former method, that is sufficient regularity of the solution guaranteeing $2-\alpha$ temporal order, and unconditional stability that allowed us to choose $\Delta x$ and $\Delta t$ of the same order. Thorough numerical calculations supporting this claim has been conducted in \cite{Plo19a} which indicate that the method based on Volterra equation is superior. 

\section{Conclusion}
We have constructed a convergent second order method for solving (\ref{eqn:VolterraY}) which can encompass self-similar solutions of a time-fractional porous medium equation on the half-line. This fact is a consequence of a series of transformations that changed a nonlocal nonlinear PDE into an ordinary Volterra integral equation. The interesting feature of the latter is a non-Lipschitz nonlinearity that possess some difficulties in numerical analysis. 

Our second order method is based on a linear (trapezoidal) reconstruction has been applied to several examples. Numerical calculations confirmed convergence with desired accuracy. We have observed that it suffices to use a relatively small number of interval subdivisions in order to obtain a decent approximation of the exact solution. Moreover, the method reproduced the asymptotic behaviour of the wetting front for large values of $m$ in three considered boundary conditions. All calculations have been conducted on a personal computer with a four core processor. Each simulation took at most few tens of seconds. This can be compared with our previous observation made in \cite{Plo19a} stating that our numerical method is much faster that the implicit finite difference scheme applied to (\ref{eqn:DiffEqPDE}). The Reader is refereed to that work for concrete computation times for these two approaches. The nonlocality, nonlinearity, stiffness, and degeneracy of the governing equation require much computing power to resolve. Transforming the original PDE into an ordinary Volterra equation reduces one degree of freedom and simplifies the free-boundary problem which facilitates the method's construction and performance. Therefore, our trapezoidal method is an accurate and fast way of computing self-similar solutions of the time-fractional porous medium equation on the half-line. 

\section*{Acknowledgement}
Ł.P. has been supported by the National Science Centre, Poland (NCN) under the grant Sonata Bis with a number NCN 2020/38/E/ST1/00153.

\small

\bibliography{biblio2}
\bibliographystyle{elsarticle-num}

\end{document}